\documentclass[12pt,reqno]{amsart}
\usepackage{amsthm, amsmath, amssymb, amscd, latexsym, multicol, verbatim, enumerate}
\usepackage[colorlinks=true, pdfstartview=FitV, linkcolor=blue, citecolor=blue, urlcolor=blue, breaklinks=true]{hyperref}

\usepackage[top=1.15in, bottom=1.15in, left=1.17in, right=1.17in]{geometry}
\usepackage[english]{babel}
\usepackage{tikz}
\usepackage{float}

\usepackage[all]{xy}
\usepackage[xcolor]{changebar}

\numberwithin{equation}{section}

\newtheorem{theorem}{Theorem}[section]
\newtheorem{corollary}[theorem]{Corollary}
\newtheorem{lemma}[theorem]{Lemma}
\newtheorem{proposition}[theorem]{Proposition}

\newtheorem{conjecture}[theorem]{Conjecture}
\newtheorem{example}[theorem]{Example}
\newtheorem{definition}[theorem]{Definition}
\newtheorem{remark}[theorem]{Remark}

\newcommand{\ba}{\mathbf{a}}
\newcommand{\bb}{\mathbf{b}}

\newcommand{\bi}{\mathbf{i}}
\newcommand{\bx}{\mathbf{x}}

\newcommand{\bz}{\mathbf{z}}

\newcommand{\bm}{\mathbf{m}}
\newcommand{\bp}{\mathbf{p}}
\newcommand{\bn}{\mathbf{n}}

\newcommand{\g}{\mathfrak{g}}
\newcommand{\n}{\mathfrak{n}}

\newcommand{\ve}{\varepsilon}

\newcommand{\ts}{\otimes}

\newcommand{\FFLV}{{\operatorname*{FFLV}}}

\title{Lusztig polytopes and FFLV polytopes}

\author{Xin Fang, Gleb Koshevoy}
\address{X. F.: Abteilung Mathematik, Department Mathematik/Informatik, Universit\"at zu K\"oln, 50931, Cologne, Germany}
\address{G. K.: The Institute for Information Transmission Problems of the RAS, 19, Bol’shoi Karetnyi per., 127051 Moscow, Russia} 
\email{xinfang.math@gmail.com}
\email{koshevoyga@gmail.com}

\begin{document}

\begin{abstract}
In this paper we prove that in type $\tt A_n$, the Feigin-Fourier-Littelmann-Vinberg (FFLV) polytope coincides with the Minkowski sum of Lusztig polytopes arising from various reduced decompositions. Using this result, we formulate a conjecture about the crystal structures on FFLV polytopes. 
\end{abstract}

\maketitle

\section{Introduction}

Constructing bases of representations of Lie algebras is one of the central topics in representation theory. For semi-simple Lie algebras and their finite dimensional irreducible representations, various bases (Gelfand-Tsetlin bases, canonical/global crystal bases, standard monomial bases, Poincar\'e-Birkhoff-Witt-type bases, Mirkovi\'c-Vilonen bases, bases arising from cluster structures, \emph{etc}) are constructed using quite different methods. Comparing these bases, or more specific, studying base change matrices, is usually a very hard question.

Each of these bases comes with a parametrisation by a polyhedral structure (polyhedral cones, convex polytopes, polyhedral complexes, \emph{etc}). The first step towards studying the base change matrices is to compare the polyhedral structures. 
For a simple Lie algebra of type $\tt A_n$ and the finite dimensional irreducible representation $V(\lambda)$ of highest weight $\lambda$, there are two PBW-type basis known for $V(\lambda)$.
\begin{enumerate}
\item The Feigin-Fourier-Littelmann-Vinberg (FFLV) basis: such a basis is compatible with the PBW filtration on $V(\lambda)$. The FFLV basis is parametrised by lattice points in the FFLV polytope - a lattice polytope having the facet description by Dyck paths.
\item The canonical basis of Lusztig: for a fixed reduced decomposition of the longest element in the Weyl group, such a basis of $V(\lambda)$ admits a parametrisation by a rational polytope, called Lusztig polytope. The facets of such polytopes are more complicated (see \cite{GKS16, GKS19} for descriptions using rhombic tilings). They are quite different to FFLV polytopes.
\end{enumerate}

For each Lusztig polytope, there exists a unique crystal structure on the set of its lattice points defined using piece-wise linear combinatorics of \cite{lusztigbook}. Such a structure is only known for lattice points in FFLV polytopes in the case where the highest weight $\lambda$ is a multiple of a fundamental weight \cite{Kus}.

The goal of this paper is to prove an unexpected relation between these two polytopes: an FFLV polytope can be written into a Minkowski sum of Lusztig polytopes associated to various reduced decompositions (Theorem \ref{Thm:Main}). 

We will present in this paper two different proofs to this result:
\begin{enumerate}
\item representation-theoretical proof: such a proof makes use of birational sequences and essential bases introduced in \cite{FaFoL} in order to compare the lattice points in both polytopes;
\item combinatorial proof: as the FFLV polytopes are defined using facet description, we compare them with the facet descriptions of the Lusztig polytopes obtained in \cite{GKS16,GKS19} using rhombic tilings or equivalently the cluster structure. In this proof, the Dyck paths bounding the FFLV polytope are identified with the paths in specific rhombic tilings.
\end{enumerate}

These two proofs are in a "polar position": a polytope can be described either as the convex hull of points or the bounded intersection of half-spaces. These two descriptions are switched by the polar duality.

This relation between these two polytopes allows us to translate the crystal structure from the Lusztig polytopes to the FFLV polytopes: for multiples of a fundamental weight, we recover the results by Kus in \cite{Kus}. In the small rank examples, when $\lambda$ is generic enough, there are more than one way to implement the crystal structure on the lattice points in the FFLV polytope. We conjecture that when the Lie algebra is of type $\tt A_n$, for a generic weight, there exists $n!$ implementations of the crystal structure on the corresponding FFLV polytope.

In Section \ref{Sec:2}, after recalling briefly the definition of the Lusztig polytopes and the FFLV polytopes, we state the main result of the paper (Theorem \ref{Thm:Main}). Two proofs of the main theorem are provided in the following two sections: in Section \ref{Sec:3} it is proved using representation theory by realising both polytopes as essential polytopes associated to a birational sequence; in Section \ref{Sec:4} a combinatorial proof is provided by explicitly writing down the facets of the Lusztig polytopes with the help of rhombic tiling. In the last Section \ref{Sec:5} we state the conjecture on the crystal structures on FFLV polytopes, and justify it in the $\tt A_2$ examples.

\subsection*{Acknowledgements} 
We thank Deniz Kus for careful reading of the manuscript. X.F. thanks Ghislain Fourier and Bea Schumann for helpful discussions. Part of this work is carried out during a research visit of X.F., he would like to thank Independent University of Moscow and the Laboratoire J.-V.Poncelet for their hospitality; G.K. thanks the Institute of Mathematics at University Cologne for hospitality and support by RSF grant 16-11-10075.

\section{Polytopes parametrising bases in Lie algebras}\label{Sec:2}

\subsection{Notations}
In this paper we fix $G=\mathrm{SL}_{n+1}(\mathbb{C})$ be the group of $(n+1)\times (n+1)$-complex matrices of determinant $1$; its Lie algebra $\g=\mathfrak{sl}_{n+1}(\mathbb{C})$ consists of traceless $(n+1)\times (n+1)$-complex matrices.

We fix the triangular decomposition $\g=\mathfrak{n}^+\oplus\mathfrak{h}\oplus\mathfrak{n}^-$ where $\mathfrak{n}^+$ (resp. $\mathfrak{n}^-$) consists of strict upper-triangular (resp. strict lower-triangular) matrices and $\mathfrak{h}$ contains the traceless diagonal matrices. Let $U^-\subseteq G$ be the subgroup of unipotent lower-triangular matrices with Lie algebra $\mathfrak{n}^-$. The corresponding universal enveloping algebras will be denoted by $U(\mathfrak{n}^-)$.

Let $n=\dim\mathfrak{h}$ be the rank of $\g$. The simple roots in $\g$ will be denoted by $\alpha_1,\cdots,\alpha_n$. Let $\Delta_+=\{\alpha_{i,j}:=\alpha_i+\cdots+\alpha_j\mid 1\leq i\leq j\leq n\}$ be the set of positive roots in $\g$ with $N=\#\Delta_+$. For $\beta\in\Delta_+$, we choose a generator $f_{\beta}$ of the root space $\mathfrak{g}_{-\beta}$. We fix $U_{-\beta}\subseteq U^-$ to be the unipotent subgroup with Lie algebra $\g_{-\beta}$. Let $\varpi_1,\cdots,\varpi_n$ be the fundamental weights and $\Lambda^+:=\sum_{i=1}^n\mathbb{N}\varpi_i$ be the set of dominant integral weights. For $\lambda\in\Lambda^+$, the finite dimensional irreducible representation of $\g$ associated to $\lambda$ will be denoted by $V(\lambda)$. We fix a highest weight vector $v_\lambda\in V(\lambda)$.

Let $W$ be the Weyl group of $\g$ with simple reflections $s_1,\cdots,s_n$, where $s_i$ corresponds to the simple root $\alpha_i$, and $w_0\in W$ be the longest element. The length function on $W$ is denoted by $\ell$. Let $R(w_0)$ be the set of all reduced decompositions of $w_0$. An element in $R(w_0)$ will be denoted by either a reduced word $\bi=(i_1,\cdots,i_N)$ or a reduced decomposition $\underline{w}_0=s_{i_1}\cdots s_{i_N}$.

We denote by $U_q(\g)$ the quantum group over $\mathbb{C}(q)$ with Chevalley generators $E_i$, $F_i$ and $K_i^{\pm 1}$ for $1\leq i\leq n$; $U_q(\mathfrak{n}^-)$ denotes the $\mathbb{C}(q)$-subalgebra of $U_q(\g)$ generated by $F_i$ for $1\leq i\leq n$. For $\lambda\in\Lambda^+$, let $V_q(\lambda)$ be the finite dimensional irreducible representation of $U_q(\g)$ of type $1$. We fix a highest weight vector $v_\lambda^q\in V_q(\lambda)$. (For readers who are not familiar with quantum groups, we recommend read \cite{Jantzen} for details.)

We will consider on $\mathbb{Z}^N$ the following total orderings: for $\ba=(a_1,\cdots,a_N)$, $\bb=(b_1,\cdots,b_N)\in\mathbb{Z}^N$,
\begin{enumerate}
\item opposite lexicographic ordering $>_{\mathrm{oplex}}$: $\ba>_{\mathrm{oplex}}\bb$ if there exists $1\leq i\leq N$ such that $a_1=b_1$, $\cdots$, $a_{i-1}=b_{i-1}$ and $a_i<b_i$;
\item right opposite lexicographic ordering $>_{\mathrm{roplex}}$: $\ba>_{\mathrm{roplex}}\bb$ if there exists $1\leq i\leq N$ such that $a_N=b_N$, $\cdots$, $a_{i+1}=b_{i+1}$ and $a_i<b_i$.
\end{enumerate}
Let $\succ$ be the partial order on $\mathbb{Z}^N$ defined by the intersection of above ordeings: $\ba\succ\bb$ if both $\ba>_{\mathrm{oplex}}\bb$ and $\ba>_{\mathrm{roplex}}\bb$ hold.

We denote $\mathbb{R}^{\Delta_+}$ the set of function from $\Delta_+$ to $\mathbb{R}$. For such a function $\ba\in\mathbb{R}^{\Delta_+}$, we write $a_\beta:=\ba(\beta)$ for $\beta\in\Delta_+$. Once an enumeration of elements in $\Delta_+$ is fixed, say $\Delta_+=\{\beta_1,\beta_2,\cdots,\beta_N\}$, we get an identification of $\mathbb{R}^{\Delta_+}$ to $\mathbb{R}^N$ sending a function $\ba$ to $(a_{\beta_1},a_{\beta_2},\cdots,a_{\beta_N})$.

For two polytopes $P$ and $Q$ in the same vector space $\mathbb{R}^N$, we denote their Minkowski sum by $P+Q$, $P+Q:=\{p+q\,|\, p\in P, \, q\in Q\}$.

\subsection{Canonical basis and Lusztig polytopes}

To a fixed reduced decomposition $\bi=(i_1,\cdots,i_N)\in R(w_0)$ we associate an enumeration of positive roots in $\Delta_+$: for $k=1,\cdots,N$, we set $\beta_k^\bi=s_{i_1}\cdots s_{i_{k-1}}(\alpha_{i_k})\in\Delta_+$, and $\underline{\beta}^\bi:=(\beta^\bi_1,\beta^\bi_2,\cdots,\beta^\bi_N)$.

For $1\leq i\leq n$, let $T_i:U_q(\g)\to U_q(\g)$ be the Lusztig's automorphism (see \cite[Chapter 37]{lusztigbook} for details, our choice here is $T_i=T_{i,1}''$).

For a reduced word $\bi\in R(w_0)$ and $m\in\mathbb{N}$, the quantum PBW root vector $F_{\beta_k^{\bi}}^{(m)}$ associated to $\beta_k^\bi$ is defined by:
$$F_{\beta_k^\bi}^{(m)}:=T_{i_1}T_{i_2}\cdots T_{i_{k-1}}(F_{i_k}^{(m)})\in U_q(\n^-),$$
where
$$F_i^{(n)}=\frac{F_i^n}{[n]_{q}!},\ \ [k]_q:=\frac{q^k-q^{-k}}{q-q^{-1}},\text{ and } [k]_q!=[k]_q[k-1]_q\cdots[1]_q.$$
For $\bm=(m_1,m_2,\cdots,m_N)\in\mathbb{N}^N$, we denote
$$F_\bi^{(\bm)}:=F_{\beta_1^{\bi}}^{(m_1)}F_{\beta_2^{\bi}}^{(m_2)}\cdots F_{\beta_N^{\bi}}^{(m_N)}\in U_q(\n^-).$$

According to \cite[Corollary 40.2.2]{lusztigbook}, for any $\bi\in R(w_0)$, the set $\{F_\bi^{(\bm)}\mid \bm\in\mathbb{N}^N\}$ forms a vector space basis of $U_q(\n^-)$.

There is a remarkable basis of $U_q(\n^-)$, whose existence is guaranteed by the following theorem (see \cite{Lus90, Cal1}).

\begin{theorem}\label{Thm:canonical}
Let $\bi\in R(w_0)$.
\begin{enumerate}
\item For any $\bm\in\mathbb{N}^N$, there exists a unique element $b_\bi(\bm)\in U_q(\n^-)$ satisfying the following properties:
$$\overline{b_\bi(\bm)}=b_\bi(\bm),$$
$$b_\bi(\bm)=F_\bi^{(\bm)}+\sum_{\bn\prec\bm}\lambda_\bm^\bn F_\bi^{(\bn)},\ \ \lambda_\bm^\bn\in q\mathbb{Z}[q].$$
\item The map $b_\bi$ sending $\bm$ to $b_\bi(\bm)$ gives a bijection between $\mathbb{N}^N$ and a basis $\mathbf{B}$ of $U_q(\n^-)$. This basis $\mathbf{B}$ does not depend on the choice of $\bi\in R(w_0)$.
\item For $\lambda\in\Lambda^+$, we set $\mathbf{B}(\lambda):=\{b\in\mathbf{B}\mid b\cdot v_\lambda^q\neq 0\}$. The set $\{b\cdot v_\lambda^q\mid b\in\mathbf{B}(\lambda)\}$ is a basis of $V_q(\lambda)$.
\end{enumerate}
\end{theorem}

This basis $\mathbf{B}$ is called the \emph{canonical basis} \cite{Lus90} (a.k.a. \emph{global crystal basis} \cite{Kas91}) of $U_q(\n^-)$, and the map $b_\bi:\mathbb{N}^N\to \mathbf{B}$ is called the Lusztig parametrisation of the canonical basis corresponding to the reduced decomposition $\bi$.

\begin{theorem}[\cite{BZ01}]
For $\bi\in R(w_0)$ and $\lambda\in\Lambda^+$, there exists a rational polytope $\mathcal{L}_\bi(\lambda)\subseteq \mathbb{R}^{\Delta_+}$ satisfying
$$b_\bi^{-1}(\mathbf{B}(\lambda))=\mathcal{L}_\bi(\lambda)_{\mathbb{Z}},$$
where $\mathcal{L}_\bi(\lambda)_{\mathbb{Z}}:=\mathcal{L}_\bi(\lambda)\cap\mathbb{Z}^{\Delta_+}$.
\end{theorem}

The polytope $\mathcal{L}_\bi(\lambda)$ is called the Lusztig polytope associated to $\bi$ and $\lambda$. The original definition of Lusztig, Berenstein-Zelevinsky \cite{Lus92, BZ01} uses piece-wise linear combinatorics arising from tropicalisation of positive maps between tori. Later we will present two different approaches to these polytopes.

There are two special reduced decompositions in $R(w_0)$:
$$\text{the lexmin decomposition: }\bi_n^{\min}=(1,2,1,3,2,1,\cdots,n,n-1,\cdots,1),$$
$$\text{the lexmax decomposition: }\bi_n^{\max}=(n,n-1,n,n-2,n-1,n,\cdots,1,\cdots,n).$$

For a dominant weight $\lambda\in\Lambda^+$, it is well-known that the Lusztig polytope $\mathcal{L}_{\bi^{\min}}(\lambda)$ is unimodular equivalent to the Gelfand-Tsetlin polytope (for recent references, see \cite{Ko, Ma}).

Recently, these polytopes are applied to study the branching problem of representations by Molev and Yakimova \cite{MY}; the tropical maximal cone of the toric degeneration of the flag variety arising from $\mathcal{L}_{\bi^{\min}}(\lambda)$ is determined by Makhlin \cite{Ma}.

\subsection{FFLV basis and FFLV polytopes}

With a different motivation, the FFLV polytopes \cite{FFL1,FFL2} appear in the study of bases compatible with the PBW filtration on finite dimensional irreducible representations of a simple Lie algebra. When the Lie algebra is of type $\mathsf{A}$ and $\mathsf{C}$, lattice points in these (lattice) polytopes parametrise such a basis in the representation.

We briefly recall the definition and basic properties of these polytopes.

A (type $\mathsf{A}$) Dyck path in $\Delta_+$ is a sequence of positive roots $\mathbf{p}=(\gamma_0,\gamma_1,\cdots,\gamma_k)$ for $k\geq 0$ satisfying 
\begin{enumerate}
\item $\gamma_0=\alpha_i$ and $\gamma_k=\alpha_j$ are simple roots;
\item if $\gamma_r=\alpha_{s,t}$, then $\gamma_{r+1}$ is either $\alpha_{s+1,t}$ or $\alpha_{s,t+1}$.
\end{enumerate}

For $1\leq i\leq j\leq n$, we set $\mathbb{P}_n:=\bigcup_{1\leq i\leq j\leq n}\mathbb{P}_{i,j}$ where $\mathbb{P}_{i,j}$ is the set of Dyck paths starting from $\alpha_i$ and ending in $\alpha_j$; 

For $\lambda=\lambda_1\varpi_1+\lambda_2\varpi_2+\cdots+\lambda_n\varpi_n\in\Lambda^+$, the polytope $\FFLV_n(\lambda)$ consists of the points $(a_\gamma)\in\mathbb{R}^{\Delta_+}$ satisfying
\begin{enumerate}
\item for any $\mathbf{p}\in\mathbb{P}_n$, if $\mathbf{p}\in\mathbb{P}_{i,j}$ then
$$\sum_{\gamma\in\mathbf{p}}a_\gamma\leq \lambda_i+\lambda_{i+1}+\cdots+\lambda_j;$$
\item for any $\gamma\in\Delta_+$, $a_\gamma\geq 0$.
\end{enumerate}

We denote $\FFLV_n(\lambda)_\mathbb{Z}:=\FFLV_n(\lambda)\cap\mathbb{Z}^{\Delta_+}$ to be the set of lattice points in $\FFLV_n(\lambda)$.

For a fixed enumeration $\beta_1,\beta_2,\cdots,\beta_N$ of $\Delta_+$ and a lattice point $\ba\in\FFLV_n(\lambda)_\mathbb{Z}$, we set

$$f^\ba:=f_{\beta_1}^{a_{\beta_1}}f_{\beta_2}^{a_{\beta_2}}\cdots f_{\beta_N}^{a_{\beta_N}}\in U(\mathfrak{n}^-).$$

\begin{theorem}[\cite{FFL1}]\label{Thm:FFL}
The following statements hold:
\begin{enumerate}
\item The set $\{f^\ba\cdot v_\lambda\mid \ba\in\FFLV_n(\lambda)\}\subseteq V(\lambda)$ forms a basis. 
\item $\FFLV_n(\lambda)$ is a lattice polytope satisfying the following Minkowski property: for $\lambda,\mu\in\Lambda^+$,
$$\FFLV_n(\lambda)+\FFLV_n(\mu)=\FFLV_n(\lambda+\mu),\ \ \FFLV_n(\lambda)_\mathbb{Z}+\FFLV_n(\mu)_\mathbb{Z}=\FFLV_n(\lambda+\mu)_\mathbb{Z}.$$
\end{enumerate}
\end{theorem}

\begin{remark}
As a consequence of the main result in \cite{Fou}, $\FFLV_n(\lambda)$ is in general not unimodular equivalent to the Lusztig polytope $\mathcal{L}_{\bi^{\min}}(\lambda)$ (since it is not unimodular equivalent to the Gelfand-Tsetlin polytope). 
\end{remark}

\subsection{Main result: statement}

We start with defining some special reduced decomposition in $R(w_0)$.

For a fixed $1\leq k\leq n$, we define
$$\underline{w}_{k,n}:=(s_ks_{k-1}\cdots s_1)(s_{k+1}s_k\cdots s_2)\cdots (s_ns_{n-1}\cdots s_{n-k+1}),$$
where $w_{k,n}$ is the corresponding element in the Weyl group. We set furthermore
$$S_{n-(k-1)+[k-1]}:=(s_ns_{n-1}\cdots s_{n-k+2})(s_ns_{n-1}\cdots s_{n-k+3})\cdots (s_{n}s_{n-1})s_n,$$
and 
$$S_{[n-k]}=(s_1s_2\cdots s_{n-k})(s_1s_2\cdots s_{n-k-1})\cdots (s_1s_2)s_1.$$
Then
$S_{n-(k-1)+[k-1]}^{-1}$ is a shift of $\bi_{k-1}^{\max}$ and $S_{[n-k]}^{-1}$ coincides with $\bi_{n-k}^{\min}$, therefore
$$\ell(w_{k,n})=(n-k+1)k,\ \ \ell(S_{n-(k-1)+[k-1]})=\frac{k(k-1)}{2},\ \ \ell(S_{[n-k]})=\frac{(n-k+1)(n-k)}{2}.$$

\begin{lemma}
For any $k=1,\cdots,n$, $\underline{w}_{k,n}\cdot S_{n-(k-1)+[k-1]}\cdot S_{[n-k]}\in R(w_0)$.
\end{lemma}

\begin{proof}
The wiring diagram corresponding to this decomposition has the following form: the first $k$ wires go parallel to the NE direction until the wire labeled $k$ touches the ``roof'', and the wires $k+1, \cdots , n+1$ go parallel to the SE direction until the wire labeled $k+1$ touches the ``floor''. Then the first $k$ wires go to the east and each two of them cross follow the lexmax reduced decomposition of the longest permutation of $\mathfrak{S}_{k-1}$ in the alphabet $\{1,\cdots,k\}$, and the wires $k+1,\cdots , n+1$ got to the east and each two of them cross follow the lexmin reduced decomposition of the inverse permutations of $\mathfrak{S}_{n-k}$ in the alphabet
$\{k+1,\cdots,n+1\}$.

We see that the intersection of the first $k$ wires with the wires labeled by $k+1,\ldots, n+1$ form a rectangular being rotated around the corner corresponding to the intersection of the $k$ and $k+1$ wires.

Here is an example with $n=5$, $k=3$.

\vskip 15pt
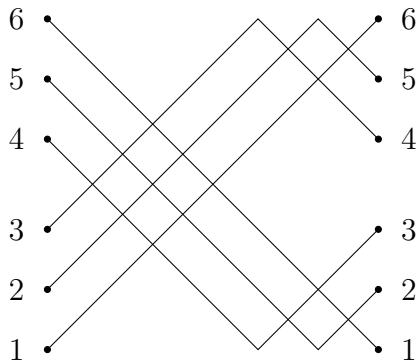
\begin{figure}[H]
\begin{tikzpicture}[scale=.4]
    \node[left] at (-0.4,0) {{$1$}};
    \node[left] at (-0.4,2) {{$2$}};
    \node[left] at (-0.4,4) {{$3$}};
    \node[left] at (-0.4,7) {{$4$}};
    \node[left] at (-0.4,9) {{$5$}};
    \node[left] at (-0.4,11) {{$6$}};
    \node[right] at (11.4,0) {{$1$}};
    \node[right] at (11.4,2) {{$2$}};
    \node[right] at (11.4,4) {{$3$}};
    \node[right] at (11.4,7) {{$4$}};
    \node[right] at (11.4,9) {{$5$}};
    \node[right] at (11.4,11) {{$6$}};
    
	\draw (0,0) -- (11,11);
	\draw (0,2) -- (9,11) -- (11,9);
	\draw (0,4) -- (7,11) -- (11,7);
	\draw (0,7) -- (7,0) -- (11,4);
	\draw (0,9) -- (9,0) -- (11,2);
	\draw (0,11) -- (11,0);
	
	\draw[fill] (0,0) circle [radius=0.1];
	\draw[fill] (0,2) circle [radius=0.1];
	\draw[fill] (0,4) circle [radius=0.1];
	\draw[fill] (0,7) circle [radius=0.1];
	\draw[fill] (0,9) circle [radius=0.1];
	\draw[fill] (0,11) circle [radius=0.1];
	\draw[fill] (11,0) circle [radius=0.1];
	\draw[fill] (11,2) circle [radius=0.1];
	\draw[fill] (11,4) circle [radius=0.1];
	\draw[fill] (11,7) circle [radius=0.1];
	\draw[fill] (11,9) circle [radius=0.1];
	\draw[fill] (11,11) circle [radius=0.1];	
\end{tikzpicture}
\caption{Wiring diagram for $n=5$ and $k=3$.}
\end{figure}
\end{proof}

We will denote $\bi^k$ the reduced word corresponding to the reduced decomposition of $w_0$ in the above lemma. For instance, when $n=3$ we have:

$$\bi^1=(1,2,3,1,2,1),\ \ \bi^2=(2,1,3,2,3,1),\ \ \bi^3=(3,2,1,3,2,3).$$

\begin{theorem}\label{Thm:Main}
For $\lambda=\lambda_1\varpi_1+\lambda_2\varpi_2+\cdots+\lambda_n\varpi_n\in\Lambda^+$, as polytopes in $\mathbb{R}^{\Delta_+}$,
$$\FFLV_n(\lambda)=\mathcal{L}_{\bi^1}(\lambda_1\varpi_1)+\mathcal{L}_{\bi^2}(\lambda_2\varpi_2)+\cdots+\mathcal{L}_{\bi^n}(\lambda_n\varpi_n).$$
Moreover, on the level of lattice points,
$$\FFLV_n(\lambda)_{\mathbb{Z}}=\mathcal{L}_{\bi^1}(\lambda_1\varpi_1)_{\mathbb{Z}}+\mathcal{L}_{\bi^2}(\lambda_2\varpi_2)_{\mathbb{Z}}+\cdots+\mathcal{L}_{\bi^n}(\lambda_n\varpi_n)_{\mathbb{Z}}.$$
\end{theorem}

\begin{remark}
When $\lambda=r\varpi_k$, Theorem \ref{Thm:Main} explains why in \cite{FF} we can get FFLV polytopes for multiples of fundamental weights from a particular chart of positive Grassmannians, although there is no connection known between FFLV bases and total positivity.
\end{remark}

In the next two sections, we will present two different proofs of this theorem:
\begin{enumerate}
\item The first representation-theoretic proof bases on realising both polytopes as Newton-Okounkov bodies. We will apply the representation-theoretic interpretation (essential monomials) of the lattice points in these Newton-Okounkov bodies given in \cite{FaFoL}. 
\item The second convex-geometric proof relies on an explicit description to the defining inequalities of the Lusztig polytope $\mathcal{L}_{\bi^k}(r\varpi_k)$ arising from an interplay of the crystal structure and the cluster structure \cite{GKS16}.
\end{enumerate}
These two proofs are in the dual position: the Newton-Okounkov theory describes lattice points in the polytope, while the cluster mirror duality gives its facets.

\section{Algebraic proof}\label{Sec:3}

\subsection{Birational sequences}

Let $S=(\beta_1,\cdots,\beta_N)$ with $\beta_i\in\Delta_+$ be a sequence of positive roots (repetitions allowed). It is called a \emph{birational sequence}, if the multiplication map 
$$U_{-\beta_1}\times U_{-\beta_2}\times \cdots \times U_{-\beta_N}\to U^-,$$
$$(u_1,u_2,\cdots,u_N)\mapsto u_1u_2\cdots u_N$$
is birational.

Let $>$ be a fixed total ordering on $\mathbb{N}^N$. We will associate to a birational sequence $S=(\beta_1,\cdots,\beta_N)$ and this total ordering a semigroup and a cone. By fixing these data one defines a filtration on $U(\mathfrak{n}^-)$ by setting for $\bm\in\mathbb{N}^N$,
$$U(\mathfrak{n}^-)_{\leq\bm}:=\mathrm{span}\{f^\ba:=f_{\beta_1}^{a_1}\cdots f_{\beta_N}^{a_N}\mid \ba\leq\bm\};$$
and similarly we have $U(\mathfrak{n}^-)_{<\bm}\subseteq U(\mathfrak{n}^-)_{\leq\bm}$. These filtrations on $U(\mathfrak{n}^-)$ induce filtrations on $V(\lambda)$ by requiring for $\bm\in\mathbb{N}^N$
$$V(\lambda)_{\leq \bm}:=U(\mathfrak{n}^-)_{\leq \bm}\cdot v_\lambda,\ \ V(\lambda)_{<\bm}:=U(\mathfrak{n}^-)_{<\bm}\cdot v_\lambda.$$

\begin{definition}[\cite{FaFoL}]\label{Defn:essential}
An element $\bm\in\mathbb{N}^N$ is called an \emph{essential exponent} with respect to $(S,>)$ if 
$$\dim\left(V(\lambda)_{\leq \bm}/V(\lambda)_{<\bm}\right)=1.$$
\end{definition}

The set of essential exponents in $V(\lambda)$ will be denoted by $\mathrm{es}_\lambda(S,>)$. We set 
$$\Gamma_\lambda(S,>):=\bigcup_{k\in\mathbb{N}}\{k\}\times\mathrm{es}_{k\lambda}(S,>)\subseteq\mathbb{N}\times\mathbb{N}^N,$$
and define the \emph{global essential monoid}
$$\Gamma(S,>):=\bigcup_{\lambda\in\Lambda^+}\{\lambda\}\times\mathrm{es}_\lambda(S,>)\subseteq\Lambda^+\times\mathbb{N}^N.$$

\begin{proposition}[\cite{FaFoL}]
With the induced structures from $\mathbb{N}\times\mathbb{N}^N$ and $\Lambda^+\times\mathbb{N}^N$, the sets $\Gamma_\lambda(S,>)$ and $\Gamma(S,>)$ are monoids. 
\end{proposition}

\subsection{Realisation of polytopes}

We set $\Lambda:=\mathbb{Z}\varpi_1+\cdots+\mathbb{Z}\varpi_n$ and $\Lambda_\mathbb{R}:=\Lambda\ts_{\mathbb{Z}}\mathbb{R}$. Let 
$$C_\lambda(S,>)\subseteq\mathbb{R}\times\mathbb{R}^N\text{  and  }C(S,>)\subseteq\Lambda_\mathbb{R}\times\mathbb{R}^N$$
be the cones generated by the sets $\Gamma_\lambda(S,>)$ and $\Gamma(S,>)$, respectively. By cutting the cone $C_\lambda(S,>)$ we obtain a convex body
$$\Delta_\lambda(S,>):=(\{1\}\times\mathbb{R}^N)\cap C_\lambda(S,>),$$
called the Newton-Okounkov body associated to $(S,>)$ and the weight $\lambda$. It is shown in \cite{FaFoL} that $\Delta_\lambda(S,>)$ coincides with the Newton-Okounkov body associated to a valuation.

We provide some examples of this construction. A reduced word $\bi=(i_1,\cdots,i_N)\in R(w_0)$ gives a birational sequence
$$P_\bi=(\beta_1^\bi,\cdots,\beta_N^\bi).$$

We fix the right opposite lexicographic ordering $>_{\mathrm{roplex}}$ on $\mathbb{Z}^N$. The following theorem identifies the Newton-Okounkov bodies associated to $(P_\bi,>_{\mathrm{roplex}})$ with the known polytopes.

By fixing the enumeration $\beta_1^\bi,\cdots,\beta_N^\bi$ of $\Delta_+$, we identify $\mathbb{R}^{\Delta_+}$ with $\mathbb{R}^N$.

\begin{theorem}[\cite{FaFoL}]\label{Thm:Lusroplex}
We have $\Delta_\lambda(P_\bi,>_{\mathrm{roplex}})=\mathcal{L}_\bi(\lambda).$
\end{theorem}

\subsection{Proof of Theorem}

We start with showing some properties of the reduced word $\bi^k$. Let $(\beta_1^{\bi^k},\beta_2^{\bi^k},\cdots,\beta_N^{\bi^k})$ be the enumeration of positive roots associated to $\bi^k$.

For a positive root $\alpha_{i,j}\in\Delta_+$, we set $\mathrm{ht}(\alpha_{i,j})=j-i+1$ to be the height of the root.

\begin{lemma}\label{Lem:betaik}
The following statements hold:
\begin{enumerate}
\item We have $\{\beta_1^{\bi^k},\cdots,\beta_{k(n-k+1)}^{\bi^k}\}=\{\alpha_{i,j}\mid 1\leq i\leq k\leq j\leq n\}$.
\item We have $\{\beta_{k(n-k+1)+1}^{\bi^k},\cdots,\beta_N^{\bi^k}\}=\{\alpha_{i,j}\mid k\notin [i,j]\}$.
\item We denote $\beta_s^{\bi^k}=\alpha_{p_1,q_1}$ and $\beta_t^{\bi^k}=\alpha_{p_2,q_2}$. If $s<t\leq k$, then either $q_1>q_2$ or $q_1=q_2$, $p_2<p_1$.
\end{enumerate}
\end{lemma}

\begin{proof}
It suffices to notice that by definition, $(\beta_1^{\bi^k},\beta_2^{\bi^k},\cdots,\beta_{k(n-k+1)}^{\bi^k})$ is
$$(\alpha_{k,k},\alpha_{k-1,k},\cdots,\alpha_{1,k},\alpha_{k,k+1},\cdots,\alpha_{1,k+1},\cdots,\alpha_{k,n},\cdots,\alpha_{1,n}).$$
\end{proof}

We first consider the case of a fundamental weight.

\begin{proposition}\label{Prop:fundamental}
We have $\FFLV_n(\varpi_k)=\mathcal{L}_{\bi^k}(\varpi_k)$.
\end{proposition}

We need some preparations for the proof:

\begin{lemma}\label{Lem:zeros}
For $\ba\in\mathcal{L}_{\bi^k}(\varpi_k)$, if $\ell>k(n-k+1)$, then $a_{\beta_\ell^{\bi^k}}=0$.
\end{lemma}

\begin{proof}
This holds by the definition of the filtration on $V(\varpi_k)$. According to Lemma \ref{Lem:betaik} (2), if $\ell>k(n-k+1)$ then $f_{\beta_\ell^{\bi^k}}$ acts by zero on the highest weight vector $v_{\varpi_k}$.
\end{proof}

We recall the description of the lattice points in the FFLV polytope in \cite{Fei, FaFoR}.

As representations of $\g$, we have $V(\varpi_k)\cong\Lambda^k\mathbb{C}^{n+1}$. We fix the standard basis $\{e_1,\cdots,e_{n+1}\}$ of $\mathbb{C}^{n+1}$. Then $V(\varpi_k)$ has a linear basis 
$$\{e_{j_1}\wedge\cdots \wedge e_{j_k}\mid 1\leq j_1<j_2<\cdots<j_k\leq n+1\},$$
where $e_1\wedge\cdots\wedge e_k$ is a highest weight vector.
We fix the index $s$ satisfying
$$1\leq j_1<\cdots<j_s\leq k< j_{s+1}<\cdots<j_k\leq n+1,$$
and set $\{p_1,\cdots,p_{k-s}\}=\{1,\cdots,k\}\setminus\{j_1,\cdots,j_s\}$ with the ordering $p_1<\cdots<p_{k-s}$.

For $\ba\in\mathbb{N}^{\Delta_+}$, we will denote $a_\ell^k:=a_{\beta_\ell^{\bi^k}}$. If $\ba$ satisfies for any $\ell>k(n-k+1)$, $a_\ell^k=0$, we will write the monomial 
$$f^{\ba}:=f_{\beta_1^{\bi^k}}^{a_1^k}f_{\beta_2^{\bi^k}}^{a_2^k}\cdots f_{\beta_k^{\bi^k}}^{a_k^k}.$$
Note that this monomial does not depend on the order of the root vectors $f_\beta$ it contains. We will write $f_{i,j}:=f_{\alpha_{i,j}}$ for short.

A monomial $f^\ba$ satisfies 
$$f^\ba\cdot e_1\wedge\cdots\wedge e_k\text{ is proportional to }e_{j_1}\wedge\cdots\wedge e_{j_k}$$ 
if and only if it has the form
\begin{equation}\label{Eq:1}
f_{p_{\sigma(1)},j_k-1}f_{p_{\sigma(2)},j_{k-1}-1}\cdots f_{p_{\sigma(k-s)},j_{s+1}-1}
\end{equation}
for some $\sigma\in\mathfrak{S}_{k-s}$.

The point in $\mathrm{FFLV}_n(\varpi_k)$ corresponding to the basis $e_{j_1}\wedge\cdots\wedge e_{j_k}$ is given by the function $\bp_{j_1,\cdots,j_k}$ defined by:
$$\bp_{j_1,\cdots,j_k}(\alpha_{i,j}):=\begin{cases}1, &\text{if }(i,j)=(p_r,j_{k-r+1}-1)\text{ for some }r\in [1,k-s]; \\ 0, & \text{otherwise}.\end{cases}$$
It corresponds to the case $\sigma=\mathrm{id}$ in \eqref{Eq:1}.

\begin{remark}
Such a function $\bp_{j_1,\cdots,j_k}$ corresponds to corners of a path from $\alpha_{1,k}$ to $\alpha_{k,n}$ in the rectangular consisting of roots in Lemma \ref{Lem:betaik} (1). 
\end{remark}

\begin{proposition}[\cite{Fei}]\label{Prop:PointsFFLV}
We have $\mathrm{FFLV}_n(\varpi_k)_{\mathbb{Z}}=\{\bp_{j_1,\cdots,j_k}\mid 1\leq j_1<\cdots<j_k\leq n+1\}.$
\end{proposition}

We turn to the proof of Proposition \ref{Prop:fundamental}.

\begin{proof}[Proof of Proposition \ref{Prop:fundamental}]
We first show that $\FFLV_n(\varpi_k)_\mathbb{Z}=\mathcal{L}_{\bi^k}(\varpi_k)_\mathbb{Z}$.
Putting together the discussions above and Theorem \ref{Thm:canonical} (3), it suffices to show that for any $1\leq j_1<\cdots<j_k\leq n+1$, $\bp_{j_1,\cdots,j_k}\in\mathcal{L}_{\bi^k}(\varpi_k)$. According to Theorem \ref{Thm:Lusroplex}, we show that $\bp_{j_1,\cdots,j_k}\in\Delta_{\varpi_k}(P_{\bi^k},>_{\mathrm{roplex}})$.

By Definition \ref{Defn:essential} and Lemma \ref{Lem:zeros}, this amount to determine under the right opposite lexicographic ordering, which monomial in \eqref{Eq:1} is minimal. According to Lemma \ref{Lem:betaik} (3), we opt to choose the root vectors $f_{i,j}$ where the second index is large, and the first index is small. In \eqref{Eq:1}, the second index
satisfies $j_k-1>j_{k-1}-1>\cdots>j_{s+1}-1$: it suffices to choose $\sigma $ such that $\sigma(1)<\cdots<\sigma({k-s})$, that is to say, $\sigma=\mathrm{id}$.

It remains to show that the polytopes are the same. Since $\FFLV_n(\varpi_k)$ is a lattice polytope, $\FFLV_n(\varpi_k)\subseteq\mathcal{L}_{\bi^k}(\varpi_k)$. It is clear that they have the same dimension. By Weyl character formula, they even share the same volume, implying the equality.
\end{proof}

Combining Proposition \ref{Prop:fundamental} and the Minkowski property in Theorem \ref{Thm:FFL} proves Theorem \ref{Thm:Main} for multiples of fundamental weights.

\begin{corollary}
For any $r\geq 1$, $\FFLV_n(r\varpi_k)=\mathcal{L}_{\bi^k}(r\varpi_k)$.
\end{corollary}

Applying again the Minkowski property in Theorem \ref{Thm:FFL} terminates the proof of Theorem \ref{Thm:Main}.

\section{Geometric proof}\label{Sec:4}

To simplify notations, we set $m=n+1$ in this section.

\subsection{Rhombic tiling}

The inequalities defining type $\mathsf{A}$ Lusztig polytopes can be described using rhombic tilings and Reineke vectors. We briefly recall these constructions following \cite{E, GKS16, GKS19}.

First draw a $2m$-gon $C_{2m}$ on the plane and fix a vertex $v_0$ of $C_{2m}$. One labels the edges of $C_{2m}$ clockwise starting from $v_0$ by $1,2,\cdots,m$ until a vertex $v_1$; these edges are called left boundary. Then continue labelling the edges starting from $v_1$ by $1,2,\cdots,m$ and call them the right boundary.

We fix a reduced decomposition $\mathbf{i}=(i_1,\cdots,i_N)\in R(w_0)$ and an enumeration of positive roots $\Delta_+=\{\beta^\bi_1,\cdots,\beta^\bi_N\}$ where $\beta_k^\bi=\alpha_{s_k,t_k}$.

We start from $\alpha_{s_1,t_1}$: this is a simple root hence $t_1=s_1$; we complete the edges on the left boundary labeled by $s_1$, $s_1+1$ to a rhombus inside of $C_{2m}$. The opposite edges in this rhombus will be labelled by the same number. This gives us a new connected set of edges labeled by $1,\cdots,m$.

We move to this new set of edges and consider the second positive root $\alpha_{s_2,t_2}$. In this new set of edges, edges labelled by $s_2$ and $t_2+1$ are neighbours, we complete them into a rhombus inside of $C_{2m}$, label the opposite edges by the same number and switch to this new set of edges. According to \cite{E}, when this procedure is applied consecutively to $\alpha_{s_1,t_1}$, $\cdots$, $\alpha_{s_N,t_N}$, we obtain a rhombic tiling $\mathcal{T}$ of the $2m$-gon $C_{2m}$. Every edge in the tiling is labeled by a number in $[m]$.

A set of edges is called \emph{connected}, if there exists one and only one path between any two vertices. A connected set of edges in $\mathcal{T}$ is called a \emph{border}, if it contains exactly one edge with each label.

Each tile $T$ in $\mathcal{T}$ has exactly two edge labels: if these edge labels are $1\leq s\neq t\leq m$, we will denote the tile by $T=[s,t]$.

A sequence $\gamma=(\gamma_i)_{1\leq i\leq r}$ of tiles in $\mathcal{T}$ is termed \emph{neighbour sequence}, if for any $1\leq i\leq r-1$, the tiles $\gamma_i$ and $\gamma_{i+1}$ share an edge.

For $1\leq t\leq m$, the $t$-strip $\mathcal{S}^t$ is defined to be the neighbour sequence $\gamma=(\gamma_i)_{1\leq i\leq m-1}$ such that for any $1\leq i\leq m-1$, one of the edges of $\gamma_i$ is labeled by $t$ and one edge labeled by $t$ in $\gamma_1$ is on the left boundary. We will denote $\mathcal{S}_k^t:=\gamma_k$.

\begin{example}
The rhombic tiling associated to the reduced word $\bi_n^{\min}$ (resp. $\bi_n^{\max}$) is called a \emph{standard} (resp. \emph{anti-standard}) tiling.

\vskip 10pt
\begin{figure}[H]
\begin{tikzpicture}[scale=.6]
    \node[left] at (-1.2,0.3) {{$1$}};
    \node[left] at (-0.8,1.4) {{\tiny [1,2]}};
    \node[left] at (-2.5,1.5) {{$2$}};
    \node[left] at (-2.6,4.3) {{$3$}};
    \node[left] at (-1.1,5.9) {{$4$}};
    \node[left] at (1.6,5.9) {{$1$}};
    \node[left] at (3.1,4.3) {{$2$}};
    \node[left] at (3.1,1.5) {{$3$}};
    \node[left] at (1.7,0.3) {{$4$}};
    \node at (-1.8,2.5) {{$1$}};
    \node at (-1.1,4.6) {{$1$}};    
    \node at (-0.5,1.3) {{$2$}};    
    \node at (0.5,3) {{$2$}};    
    \node at (-0.5,3) {{$3$}};    
    \node at (0.5,1.3) {{$3$}};    
    \node at (1.1,4.6) {{$4$}};    
    \node at (1.8,2.5) {{$4$}};    

    \node[left] at (-0.8,3.4) {{\tiny [1,3]}};
    \node[left] at (0.5,4.9) {{\tiny [1,4]}};
    \node[left] at (0.7,2.1) {{\tiny [2,3]}};
    \node[left] at (2.2,3.3) {{\tiny [2,4]}};
    \node[left] at (2.2,1.4) {{\tiny [3,4]}};
	\draw (0,0) -- (-2,1) -- (-3,3) -- (-2,5) -- (0,6) -- (2,5) -- (3,3) -- (2,1) -- (0,0);
	\draw (0,0) -- (-1,2) -- (0,4) -- (1,2) -- (0,0);
	\draw (-1,2) -- (-3,3);
	\draw (0,4) -- (-2,5);
	\draw (1,2) -- (3,3);
	\draw (0,4) -- (2,5);
\end{tikzpicture}
\caption{rhombic tiling for $\bi_3^{\min}$.}\label{fig:min}
\end{figure}
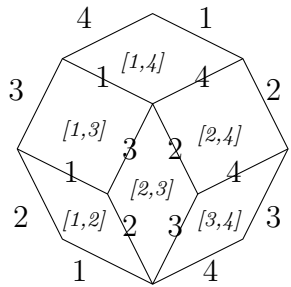
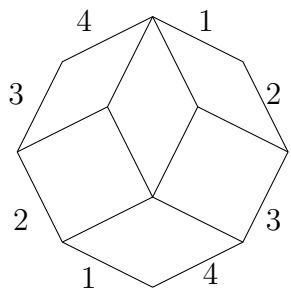
\begin{figure}[H]
\begin{tikzpicture}[scale=.6]
    \node[left] at (-1.0,0.2) {{$1$}};
    \node[left] at (-2.5,1.5) {{$2$}};
    \node[left] at (-2.6,4.3) {{$3$}};
    \node[left] at (-1.1,5.9) {{$4$}};
    \node[left] at (1.6,5.9) {{$1$}};
    \node[left] at (3.1,4.3) {{$2$}};
    \node[left] at (3.1,1.5) {{$3$}};
    \node[left] at (1.7,0.3) {{$4$}};
    
	\draw (0,0) -- (-2,1) -- (-3,3) -- (-2,5) -- (0,6) -- (2,5) -- (3,3) -- (2,1) -- (0,0);
	\draw (0,6) -- (-1,4) -- (0,2) -- (1,4) -- (0,6);
	\draw (-1,4) -- (-3,3);
	\draw (0,2) -- (-2,1);
	\draw (1,4) -- (3,3);
	\draw (0,2) -- (2,1);
\end{tikzpicture}
\caption{rhombic tiling for $\bi_3^{\max}$.}\label{fig:max}
\end{figure}
\end{example}

For any $s\in[2m]$ we define a partial order $\preceq_s$ on the tiles in $\mathcal{T}$ in the following way: we label the boundary of $C_{2m}$ starting from $v_0$ by $b_1,\cdots,b_{2m}$. Let $B_1$ be the border consisting of edges $b_{m+s+1},\cdots,b_{2m+s}$ where the indices are understood modulo $2m$.

Denote $\mathcal{T}_1^s$ to be the set of tiles in $\mathcal{T}$ intersecting $B_1$ in two edges. We move to a new border $B_2$ obtained from $B_1$ by: for every tile in $\mathcal{T}_1^s$, replace the two edges intersecting $B_1$ by the other two edges. Then denote $\mathcal{T}^s_2$ to be the set of tiles in $\mathcal{T}\setminus\mathcal{T}_1^s$ intersecting $B_2$ with two edges and repeat the above procedure.

Eventually one obtain a partition of $\mathcal{T}$ into a disjoint union of $\mathcal{T}_1^s$, $\mathcal{T}_2^s$, $\cdots$. 

For two tiles $T,T'\in\mathcal{T}$ such that $T\in\mathcal{T}_t^s$ and $T'\in\mathcal{T}_r^s$, we will denote $T\preceq_s T'$ if $t\leq r$.

For $s\in [2m]$ a neighbour sequence $\gamma=(\gamma_1,\cdots,\gamma_m)$ in $\mathcal{T}$ is called \emph{$s$-ascending}, if for any $1\leq i\leq m-1$, $\gamma_i\prec_s\gamma_{i+1}$.

An $s$-ascending neighbour sequence $(\gamma_1,\cdots,\gamma_p)$ in $\mathcal{T}$ is called an \emph{$s$-crossing}, if $\gamma_1=\mathcal{S}_1^s$ and $\gamma_p=\mathcal{S}_1^{s+1}$. 

To an $s$-ascending neighbour sequence $(\gamma_1,\cdots,\gamma_p)$ in $\mathcal{T}$ we associate a \emph{strip sequence} $(s_1,\cdots,s_r)$ in the following way: there exists $1\leq j_1<\cdots<j_r<j_{r+1}=p$ such that $\gamma_1,\cdots,\gamma_{j_1}\in \mathcal{S}^{s_1}$, $\gamma_{j_1+1},\cdots,\gamma_{j_2}\in \mathcal{S}^{s_2}$, $\cdots$, $\gamma_{j_r+1},\cdots,\gamma_{p}\in \mathcal{S}^{s_r}$. 

An $s$-crossing is uniquely determined by its strip sequence.

Let $\Gamma_s$ denote the set of $s$-crossings in $\mathcal{T}$. We denote $\mathcal{W}_s$ the $s$-crossings given by the strip sequence $(s,s+1)$. Such an $s$-crossing exists, and will be called an $s$-\emph{comb}.

\subsection{Dual Reineke vectors and H-description of Lusztig polytopes}

We introduce a dual version of the constructions above. They give the potential facets of Lusztig polytopes.

The set $\Gamma_s^*$ of dual $s$-crossings consists of $(m+s)$-ascending neighbour sequences $(\gamma_1,\cdots,\gamma_p)$ at $\gamma_1=\mathcal{S}_{n}^s$ and ending at $\gamma_p=\mathcal{S}_{n}^{s+1}$. One can similarly define the strip sequence of a dual $s$-crossing. A dual $s$-crossing is called a dual $s$-comb, if its strip sequence is $(s,s+1)$.

A dual $s$-crossing $\gamma=(\gamma_1,\cdots,\gamma_p)\in\Gamma_s$ is called a \emph{dual Reineke $s$-crossing}, if for any $\gamma_i=[a,b]$ such that $\gamma_{i-1},\gamma_i,\gamma_{i+1}$ lie in the same strip $\mathcal{S}^a$, \begin{enumerate}
\item[-] if $b\leq s$ then $a>b$;
\item[-] if $b\geq s+1$ then $a<b$.
\end{enumerate}
Let $\mathcal{R}_s^*\subseteq\Gamma_s^*$ denote the set of dual Reineke $s$-crossings.

Let $\gamma=(\gamma_1,\cdots,\gamma_p)\in\Gamma_s$ with strip sequence $(s_1=s,\cdots,s_q=s+1)$, we define 
$$
(\mathsf{r}(\gamma))_T:=\begin{cases}
\mathrm{sgn}(s_{i+1}-s_i), & \text{if }T=[s_i,s_{i+1}]\text{ for some }1\leq i\leq q-1;\\
0, & \text{otherwise};
\end{cases}
$$
$$
\ve_s(T):=\begin{cases}
1, & \text{if }T=[a,b]\in\mathcal{T} \text{ and }a\leq s<s+1\leq b;\\
-1 & \text{else};
\end{cases}
$$
and for $\mathbf{x}\in\mathbb{R}^\mathcal{T}$,
$$\mathsf{s}(\gamma)(\mathbf{x})=\sum_{T\in\Gamma,\ \ve_s(T)=1}x_T-\sum_{\substack{T\in\gamma,\ \ve_s(T)=-1,\\ (\mathsf{r}(\gamma))_T=0}}x_T.$$

\begin{theorem}[\cite{GKS19}]\label{GKS19}
For $\lambda=\lambda_1\varpi_1+\cdots+\lambda_n\varpi_n\in\Lambda^+$, 
$$\mathcal{L}_\bi(\lambda)=\left\{\mathbf{x}\in\mathbb{R}_{\geq 0}^{\Delta_+}\mid \text{for any }s\in [m],\ \gamma\in\mathcal{R}_s^*,\ \mathsf{s}(\gamma)(\mathbf{x})\leq\lambda_s\right\}.$$
\end{theorem}


\subsection{Proof of Theorem}

The goal of this subsection is to give a second proof to Theorem \ref{Thm:Main}. 

We start from considering the case $\lambda=\varpi_k$ for $1\leq k< m=n+1$ and the rhombic tiling associated to $\bi^k$. Such a tiling can be obtained by gluing together the following three parts:
\begin{enumerate}
\item[-] a rectangular tableau of size $k\times (m-k)$ slightly rotated counterclockwise around its SW-corner;
\item[-] the east border of the rectangle is glued with the left border of the anti-standard tiling for $\mathrm{SL}_k$, where the top vertex in the anti-standard tiling is glued together with the NE-corner vertex of the rectangle; we denote the tiles in the anti-standard tiling by $\mathcal{T}'$;
\item[-] the south border of the rectangular tableau is glued with the left border of the standard tiling for $\mathrm{SL}_{m-k}$, where the bottom vertex in the standard tiling is glued together with the SW-corner vertex of the rectangle; we denote the tiles in the standard tiling by $\mathcal{T}''$.
\end{enumerate}

Here is an example for $m=7$ and $k=3$:

\begin{figure}[H]
\begin{tikzpicture}[scale=.4]
   
   \node at (-2,0) {{$1$}};    
   \node at (-5.7,1.2) {{$2$}};    
   \node at (-8.2,4) {{$3$}};    
   \node at (-9.8,7) {{$4$}};    
   \node at (-9.3,11) {{$5$}};    
   \node at (-6.8,13.5) {{$6$}};    
   \node at (-3,15.2) {{$7$}};    
   \node at (1.8,15) {{$1$}};    
   \node at (4.8,13.5) {{$2$}};    
   \node at (7.4,11) {{$3$}};    
   \node at (8.5,7.2) {{$4$}};    
   \node at (8,4) {{$5$}};    
   \node at (6,1.5) {{$6$}};    
   \node at (2.2,-0.1) {{$7$}};

	\draw (0,0) -- (-4,1) -- (-7,3) -- (-9,6) -- (-9,9) -- (-8,12) -- (-5,14) -- (-1,15) -- (3,14) -- (6,12) -- (8,9) -- (8,6) -- (7,3) -- (4,1) -- (0,0);
	\draw (0,0) -- (0,3) -- (1,6) -- (4,8) -- (8,9);
	\filldraw[fill=black,fill opacity=0.1](0,0) -- (0,3) -- (1,6) -- (4,8) -- (8,9) -- (8,6) -- (7,3) -- (4,1) -- (0,0);
	\filldraw[fill=black,fill opacity=0.1](-1,15) -- (1,12) -- (4,10) -- (8,9) -- (6,12) -- (3,14) -- (-1,15);
	\draw (-4,1) -- (-4,4) -- (-3,7) -- (0,9) -- (4,10) -- (8,9);
	\draw (-7,3) -- (-7,6) -- (-6,9) -- (-3,11) -- (1,12) -- (4,10);
	\draw (0,3) -- (-4,4) -- (-7,6) -- (-9,9);
	\draw (1,6) -- (-3,7) -- (-6,9) -- (-8,12);
	\draw (4,8) -- (0,9) -- (-3,11) -- (-5,14);
	\draw (0,0) -- (3,2) -- (7,3);
	\draw (0,0) -- (1,3) -- (4,5) -- (8,6);
	\draw (1,12) -- (-1,15);
	\draw (4,8) -- (4,5) -- (3,2);
	\draw (1,6) -- (1,3);
	\draw (-1,15) -- (2,13) -- (4,10);
	\draw (2,13) -- (6,12);

\end{tikzpicture}
  \caption{Rhombic tiling for $m=7$ associated to $\bi^3$.}
  \label{fig:sl7}
\end{figure}
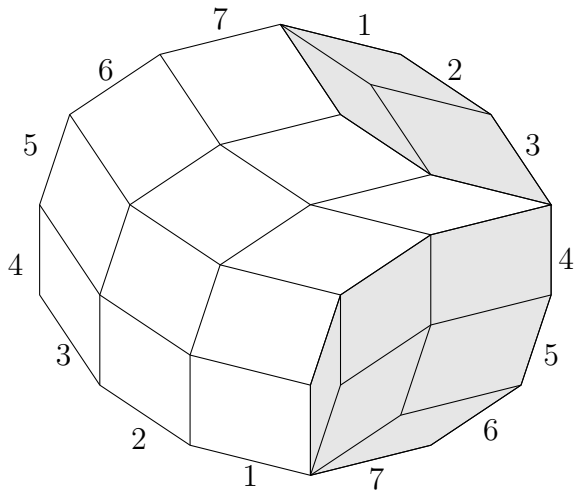

In the following we will denote the tile $[s,t]$ by $T_{[s,t]}$; for $\bx\in\mathbb{R}^{\mathcal{T}}$, the value assigned to the tile $[s,t]$ is $x_{s,t}$.

\begin{lemma}\label{Lem:A}
For $\bx\in\mathcal{L}_{\bi^k}(r\varpi_k)$, if $a\leq k$ and $k+1\leq b$ do not hold simultaneously, then $x_{a,b}=0$.
\end{lemma}

\begin{proof}
The tiles in the anti-standard (resp. standard) part $\mathcal{T}'$ (resp. $\mathcal{T}''$) have the form $T_{[a,b]}$ where $a,b<k$ (resp. $a,b\geq k+1$). 

We look at the set $\Gamma_s$ for $s<k$. First notice that the dual $s$-crossings will not go outside of the tiles in $\mathcal{T}'$, and any tile in $\mathcal{T}'$ is contained in some dual Reineke $s$-crossing for $s<k$. By Theorem \ref{GKS19}, $\lambda_s=0$ implies that for any $a\leq b<k$, $x_{a,b}=0$.

A similar argument shows that for any $k+1\leq a\leq b$, $x_{a,b}=0$.
\end{proof}

We consider the dual $k$-crossings $\Gamma_k^*$. The dual $k$-comb is the union of the $k$-strip $\mathcal{S}^k$ and the $(k+1)$-strip $\mathcal{S}^{k+1}$; it turns from the $k$-strip to the $(k+1)$-strip at the tile $T_{[k,k]}$.

\begin{lemma}\label{Lem:B}
The following statements hold:
\begin{enumerate}
\item Each dual $k$-crossing is contained in the dual $k$-comb.
\item We have $\Gamma_k^*=\mathcal{R}_k^*$.
\end{enumerate}
\end{lemma}

\begin{proof}
\begin{enumerate}
\item This follows from the dual version of the poset structure on $\Gamma_k^*$, where the dual $k$-comb is the maximal element \cite[Section 2.4]{GKS16}.
\item Let $\gamma_i=[a,b]$ be contained in a $k$-crossing such that $\gamma_{i-1},\gamma_{i+1}$ are both in the $a$-strip. First assume that $b\leq k$, such a $b$-strip is horizontal, to cross it the $a$-strip must be vertical, hence $a>b$. Similarly if $b\geq k+1$, such a $b$-strip is vertial, hence $a<b$.
\end{enumerate}
\end{proof}

We consider a dual $k$-crossing in $\Gamma_k^*$: such a crossing starts from the right boundary at the tile having a boundary labeled by $k$. Before going to the rectangular tableau, it goes along the tiles $T_{[1,k]}$, $\cdots$, $T_{[r,k]}$ for some $r\leq k-1$. Then the dual $k$-crossing goes into the rectangular tableau at one of the tiles $T_{[1,m]}$, $\cdots$, $T_{[k,m]}$. 

Inside the rectangular tableau, when the crossing reaches a tile $T_{[p,q]}$, the next tile in the crossing can only be $T_{[p-1,q]}$ or $T_{[p,q-1]}$. The crossing goes until it reaches the $1$-strip, say, tiles $T_{[1,k+1]}$, $\cdots$, $T_{[1,m]}$. After that the crossing goes into the standard tiling, along the $(k+1)$-strip until it reaches the tile having a boundary labeled by $k+1$.

As a conclusion, inside the rectangular tableau, a dual $k$-crossing gives a Dyck path, which is saturated in the sense that there is no such a Dyck path in the rectangular tableau containing it.

It remains to consider $\mathsf{s}(\gamma)$ for $\gamma\in\mathcal{R}_k^*$. Every tile $T$ in the rectangular tableau gets $\ve_k(T)=1$. We do not need to care about $\ve_k(T)$ for $T$ being outside of the rectangular tableau since the corresponding coordinates are always zero by Lemma \ref{Lem:A}. Therefore the Lusztig polytope $\mathcal{L}_{\bi^k}(r\varpi_k)$ admits the following description:
\begin{enumerate}
\item[-] for any $r\leq s$ such that $k\notin [r,s]$, $x_{r,s}=0$;
\item[-] for any Dyck path starting from the right border of the rectangular tableau and ending up with the bottom of the tableau, the sum of the coordinates associated to the tiles is less or equal to $r$.
\end{enumerate}
These are nothing but the defining inequalities of $\mathrm{FFLV}(r\varpi_k)$. The proof is complete.

\section{Crystal structure}\label{Sec:5}
As an application to the main result, we propose a conjecture on crystal structures on FFLV parametrisations and examine it in small rank examples.

\subsection{Crystal structure on Lusztig polytopes}
Recall (see Lemma \ref{Lem:betaik}) that for the reduced decomposition $\mathbf i^k$, the corresponding enumeration of positive roots
begins with $k\times (n-k+1)$  roots $\alpha_{i,j}$, $1\le i\le k\le j\le n$.

For $1\leq a\leq n$, let $f_a$ denote the Kashiwara operator corresponding to $a$. We will denote $f_{a,k}$ the Kashiwara operator for $\mathcal{L}_{\bi^k}(r\varpi_k)$ to emphasise its connection to the fundamental weight $\varpi_k$.

The crystal structure on lattice points of the Lusztig polytope $\mathcal L_{\mathbf i^k}(r\varpi_k)$ is defined by the set of Reineke vectors in \cite[Section 4]{GKS16}. 

Precisely, for a lattice point $\bx\in\mathcal{L}_{\bi^k}(r\varpi_k)$, the point $f_{a,k}(\bx)$ takes one of the following forms:
\begin{enumerate}
\item when $1\leq a<k$, there exists $k\leq j\leq n$ such that $f_{a,k}(\bx)=\bx-\delta_{a+1,j}+\delta_{a,j}$;
\item when $k<a\leq n$, there exists $1\leq i\leq k$ such that $f_{a,k}(\bx)=\bx-\delta_{i,a-1}+\delta_{i,a}$;
\item when $a=k$, $f_{a,k}(\bx)=\bx+\delta_{k,k}$;
\end{enumerate}
where $\delta_{i,j}$ is the function in $\mathbb{R}^{\Delta_+}$ taking value $1$ on $\alpha_{i,j}$ and $0$ on the other positive roots.

As a consequence of Theorem \ref{Thm:Main}, on $\FFLV_n(r\varpi_k)$ there exists a crystal structure. Such a structure coincides with the one defined explicitly by Kus in \cite{Kus}. The main results in \cite[Section 3]{Kus} then follow from the crystal structures on Lusztig polytopes described above.

\subsection{Crystal structures on FFLV polytopes}
For $\lambda\in\Lambda^+$, we define an edge-colored directed graph structure on the set $\FFLV_n(\lambda)_\mathbb{Z}$ with colors $\{1,2,\cdots,n\}$. For a point $\bz\in\FFLV_n(\lambda)_\mathbb{Z}$ and $1\leq a\leq n$, for $k=1,\cdots,n$, there exists an edge colored by $a$ from $\bz$ to $f_{a,k}(\bz)$ if $f_{a,k}(\bz)$ is a lattice points in $\FFLV_n(\lambda)_\mathbb{Z}$.

We denote this colored directed graph by $\mathrm{PB}_n(\lambda)$. Below is an example of $\mathrm{PB}_3(\varpi_1+\varpi_2)$, where the color $1$ (resp. $2$) is displayed by red (resp. blue). The coordinates $e_1,e_{12},e_2$ stand for the functions in $\mathbb{R}^{\Delta_+}$ corresponding to $\alpha_1,\alpha_1+\alpha_2,\alpha_2$, respectively.

\begin{figure}[H]
\begin{tikzpicture}[scale=.45, arr/.style={thick,->,>=stealth},dsh/.style={thick,densely dashed},]
   
    \node at (-7,-6) {{$e_2$}};    
    \node at (1,8) {{$e_1$}};
    \node at (11,1) {{$e_{12}$}};
	\draw [dsh] (0,0) -- (0,4);
	\draw [dsh] (0,0) -- (8,0);
	\draw [dsh] (0,0) -- (-2.8,-2.8);
	\draw [arr] (8,0) -- (11,0);
	\draw [arr] (0,4) -- (0,8);
	\draw [arr] (-2.8,-2.8) -- (-7,-7);
	\draw (0,4) -- (4,4) -- (-2.8,1.2) -- (0,4);
	\draw (-2.8,1.2) -- (-2.8,-2.8) -- (1.2,-2.8) -- (-2.8,1.2);
	\draw (4,4) -- (8,0) -- (1.2,-2.8);
	
	\draw[fill] (0,0) circle [radius=0.1];	
	\draw[fill] (0,4) circle [radius=0.1];	
	\draw[fill] (4,0) circle [radius=0.1];	
	\draw[fill] (4,4) circle [radius=0.1];	
	\draw[fill] (-2.8,1.2) circle [radius=0.1];	
	\draw[fill] (-2.8,-2.8) circle [radius=0.1];	
	\draw[fill] (8,0) circle [radius=0.1];	
	\draw[fill] (1.2,-2.8) circle [radius=0.1];	
	\draw [red,arr,opacity=0.9] (-2.8,-2.8) -- (-2.8,1.2);
	\draw [red,arr,opacity=0.9] (0,0) -- (0,4);
	\draw [red,arr,opacity=0.9] (4,0) -- (4,4);
	\draw [red,arr,opacity=0.9] (1.2,-2.8) -- (8,0);

	\draw [blue,arr,opacity=0.9] (0,0) -- (-2.8,-2.8);
	\draw [blue,arr,opacity=0.9] (-2.8,1.2) -- (1.2,-2.8);
	\draw [blue,arr,opacity=0.9] (0,4) -- (4,0);
	\draw [blue,arr,opacity=0.9] (4,4) -- (8,0);
	
	\draw [red,arr,opacity=0.9] (0,0) -- (0,4);
	\draw [red,arr,opacity=0.9] (-2.8,1.2) -- (4,4);
	\draw [red,arr,opacity=0.9] (-2.8,-2.8) -- (4,0);
	\draw [red,arr,opacity=0.9] (1.2,-2.8) -- (8,0);

	\draw [blue,arr,opacity=0.9] (0,0) -- (-2.8,-2.8);
	\draw [blue,arr,opacity=0.9] (0,4) -- (-2.8,1.2);
	\draw [blue,arr,opacity=0.9] (4,0) -- (1.2,-2.8);
	\draw [blue,arr,opacity=0.9] (4,4) -- (8,0);

\end{tikzpicture}
  \caption{The directed graph $\mathrm{PB}_3(\varpi_1+\varpi_2)$.}
\end{figure}
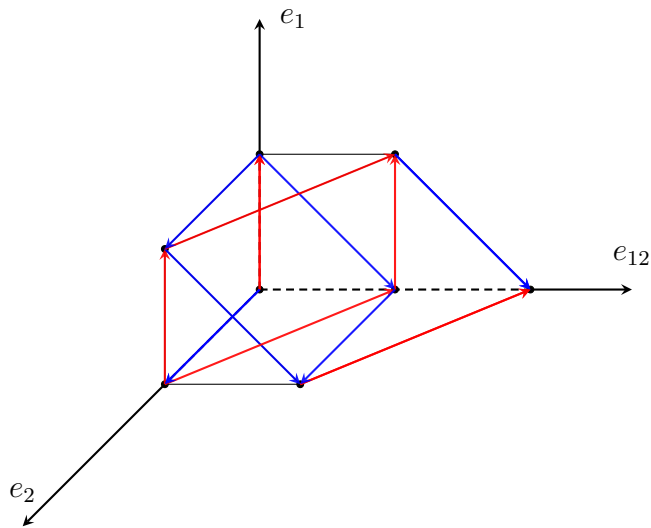

An edge-colored directed graph with vertices $\FFLV_n(\lambda)_\mathbb{Z}$ is called an \emph{FFLV-crystal graph}, if
\begin{itemize}
\item at each vertex, for a fixed color, it has at most one of the emanating edges in $\mathrm{PB}_n(\lambda)$ having this color;
\item the Stembridge axioms in \cite{St} (or equivalently, local conditions (A1)--(A4) in \cite{a2}) for any pair of colors $a$ and $b$ with $|a-b|=1$ are satisfied.
\end{itemize}

\begin{conjecture}
When $\lambda$ is regular, there exist $n!$ FFLV-crystal graphs on $\FFLV_n(\lambda)_{\mathbb{Z}}$. Each of crystal graphs is defined by fixing an element in $\mathfrak{S}_n$.
\end{conjecture}

To be precise, this choice of $\sigma\in\mathfrak{S}_n$ defines a total ordering on the fundamental weights by setting 
$$\varpi_{\sigma(1)}>\varpi_{\sigma(2)}>\cdots>\varpi_{\sigma(n)}.$$

Such a choice of the total ordering means that there exists an iterative process such that at each vertex $\bz\in\FFLV_n(\lambda)_{\mathbb{Z}}$, if there are several, allowed due to this process, edges in $\mathrm{PB}_n(\lambda)$ emanating from this vertex to $f_{a,k_1}(\bz),\cdots,f_{a,k_r}(\bz)$, we choose the edge towards $f_{a,k}(\bz)$ where $k$ is the maximal element in $\{k_1,\cdots,k_r\}$ with respect to the fixed total ordering. Our refined conjecture affirms that after making such choices for all vertices in $\FFLV_n(\lambda)_{\mathbb{Z}}$, there exists a unique FFLV-crystal graph.

\subsection{Small rank cases}
We study this conjecture in the case of $\mathrm{SL}_3$. In this case $\bi^1=(1,2,1)$ and $\bi^2=(2,1,2)$. Let $\lambda=a\varpi_1+b\varpi_2\in\Lambda^+$. We will describe two crystal graphs on $\mathrm{FFLV}_2(\lambda)_{\mathbb{Z}}$. 

First we choose $\sigma=\mathrm{id}\in\mathfrak{S}_2$, it corresponds to a total ordering $\varpi_1>\varpi_2$. 

Let $B^{>}(a, b)$ be the edge-colored directed graph on lattice points of $\FFLV_2(a\varpi_1+b\varpi_2)_\mathbb{Z}$ such that its monochromatic paths are defined in the following way:

\begin{enumerate}
\item[(1)] For the color $1$, we take the ''sky'' paths depicted below and their translations by vectors $(-k, k, 0)$, $k=1,\ldots$ (precisely we takes the parts of translated paths which belong to $\FFLV_2(\lambda)$).
\item[(2)] For the color $2$, we take the ''ground'' paths depicted below and their
translations by vectors $(k, 0, k)$,  $k=1, \ldots$ (precisely we takes the parts of translated paths which belong to $\FFLV_2(\lambda)$).
\end{enumerate}

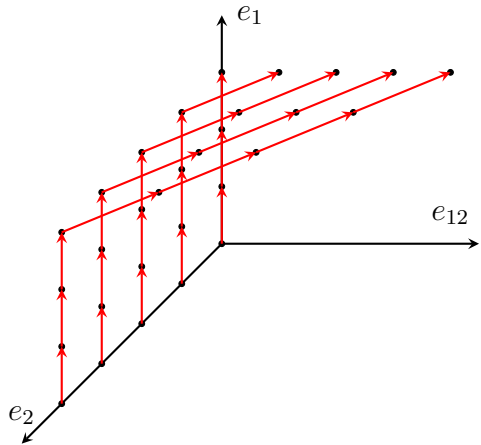
\begin{figure}[H]
\begin{tikzpicture}[scale=.38, arr/.style={thick,->,>=stealth},]
   
    \node at (-7,-6) {{$e_2$}};    
    \node at (1,8) {{$e_1$}};
    \node at (8,1) {{$e_{12}$}};

	\draw [arr] (0,0) -- (9,0);
	\draw [arr] (0,0) -- (0,8);
	\draw [arr] (0,0) -- (-7,-7);
	\draw[fill] (0,0) circle [radius=0.1];	
	\draw[fill] (0,2) circle [radius=0.1];	
	\draw[fill] (0,4) circle [radius=0.1];	
	\draw[fill] (0,6) circle [radius=0.1];	
	\draw[fill] (-1.4,-1.4) circle [radius=0.1];
	\draw[fill] (-2.8,-2.8) circle [radius=0.1];
	\draw[fill] (-4.2,-4.2) circle [radius=0.1];
	\draw[fill] (-5.6,-5.6) circle [radius=0.1];
	\draw[fill] (-1.4,0.6) circle [radius=0.1];
	\draw[fill] (-2.8,-0.8) circle [radius=0.1];
	\draw[fill] (-4.2,-2.2) circle [radius=0.1];
	\draw[fill] (-5.6,-3.6) circle [radius=0.1];
	\draw[fill] (-1.4,2.6) circle [radius=0.1];
	\draw[fill] (-2.8,1.2) circle [radius=0.1];
	\draw[fill] (-4.2,-0.2) circle [radius=0.1];
	\draw[fill] (-5.6,-1.6) circle [radius=0.1];
	\draw[fill] (-1.4,4.6) circle [radius=0.1];
	\draw[fill] (-2.8,3.2) circle [radius=0.1];
	\draw[fill] (-4.2,1.8) circle [radius=0.1];
	\draw[fill] (-5.6,0.4) circle [radius=0.1];
	\draw[fill] (2,6) circle [radius=0.1];
	\draw[fill] (4,6) circle [radius=0.1];
	\draw[fill] (6,6) circle [radius=0.1];
	\draw[fill] (8,6) circle [radius=0.1];
	\draw[fill] (0.6,4.6) circle [radius=0.1];
	\draw[fill] (2.6,4.6) circle [radius=0.1];
	\draw[fill] (4.6,4.6) circle [radius=0.1];
	\draw[fill] (-0.8,3.2) circle [radius=0.1];
	\draw[fill] (1.2,3.2) circle [radius=0.1];
	\draw[fill] (-2.2,1.8) circle [radius=0.1];

	\draw [red,arr] (-1.4,-1.4) -- (-1.4,0.6);
	\draw [red,arr] (-1.4,0.6) -- (-1.4,2.6);
	\draw [red,arr] (-1.4,2.6) -- (-1.4,4.6);
	\draw [red,arr] (-2.8,-2.8) -- (-2.8,-0.8);
	\draw [red,arr] (-2.8,-0.8) -- (-2.8,1.2);
	\draw [red,arr] (-2.8,1.2) -- (-2.8,3.2);
	\draw [red,arr] (-4.2,-4.2) -- (-4.2,-2.2);
	\draw [red,arr] (-4.2,-2.2) -- (-4.2,-0.2);
	\draw [red,arr] (-4.2,-0.2) -- (-4.2,1.8);
	\draw [red,arr] (-5.6,-5.6) -- (-5.6,-3.6);
	\draw [red,arr] (-5.6,-3.6) -- (-5.6,-1.6);
	\draw [red,arr] (-5.6,-1.6) -- (-5.6,0.4);
	\draw [red,arr] (-1.4,4.6) -- (2,6);
	\draw [red,arr] (-2.8,3.2) -- (0.6,4.6);
	\draw [red,arr] (0.6,4.6) -- (4,6);
	\draw [red,arr] (-4.2,1.8) -- (-0.8,3.2);
	\draw [red,arr] (-0.8,3.2) -- (2.6,4.6);
	\draw [red,arr] (2.6,4.6) -- (6,6);
	\draw [red,arr] (-5.6,0.4) -- (-2.2,1.8);
	\draw [red,arr] (-2.2,1.8) -- (1.2,3.2);
	\draw [red,arr] (1.2,3.2) -- (4.6,4.6);
	\draw [red,arr] (4.6,4.6) -- (8,6);
	\draw [red,arr] (0,0) -- (0,2);
	\draw [red,arr] (0,2) -- (0,4);
	\draw [red,arr] (0,4) -- (0,6);

\end{tikzpicture}
  \caption{''Sky'' crystal paths of color $1$, $a=3$ and $b=4$.}
\end{figure}

\begin{figure}[H]
\begin{tikzpicture}[scale=.38, arr/.style={thick,->,>=stealth},]
   
    \node at (-7,-6) {{$e_2$}};    
    \node at (1,8) {{$e_1$}};
    \node at (16,1) {{$e_{12}$}};

	\draw [arr] (0,0) -- (16,0);
	\draw [arr] (0,0) -- (0,8);
	\draw [arr] (0,0) -- (-7,-7);
	\draw[fill] (0,0) circle [radius=0.1];	
	\draw[fill] (0,2) circle [radius=0.1];	
	\draw[fill] (0,4) circle [radius=0.1];	
	\draw[fill] (0,6) circle [radius=0.1];	
	
	\draw[fill] (2,0) circle [radius=0.1];	
	\draw[fill] (4,0) circle [radius=0.1];	
	\draw[fill] (6,0) circle [radius=0.1];	
	\draw[fill] (8,0) circle [radius=0.1];	
	\draw[fill] (10,0) circle [radius=0.1];	
	\draw[fill] (12,0) circle [radius=0.1];	
	\draw[fill] (14,0) circle [radius=0.1];	
	
	\draw[fill] (2,2) circle [radius=0.1];	
	\draw[fill] (4,2) circle [radius=0.1];	
	\draw[fill] (6,2) circle [radius=0.1];	
	\draw[fill] (8,2) circle [radius=0.1];	
	\draw[fill] (10,2) circle [radius=0.1];	
	\draw[fill] (12,2) circle [radius=0.1];	
	
	\draw[fill] (2,4) circle [radius=0.1];	
	\draw[fill] (4,4) circle [radius=0.1];	
	\draw[fill] (6,4) circle [radius=0.1];	
	\draw[fill] (8,4) circle [radius=0.1];	
	\draw[fill] (10,4) circle [radius=0.1];	
	
	\draw[fill] (2,6) circle [radius=0.1];	
	\draw[fill] (4,6) circle [radius=0.1];	
	\draw[fill] (6,6) circle [radius=0.1];	
	\draw[fill] (8,6) circle [radius=0.1];	
	
	\draw[fill] (0.6,-1.4) circle [radius=0.1];	
	\draw[fill] (2.6,-1.4) circle [radius=0.1];	
	\draw[fill] (4.6,-1.4) circle [radius=0.1];	
	\draw[fill] (6.6,-1.4) circle [radius=0.1];	
	\draw[fill] (8.6,-1.4) circle [radius=0.1];	
	\draw[fill] (10.6,-1.4) circle [radius=0.1];	

	\draw[fill] (-0.8,-2.8) circle [radius=0.1];	
	\draw[fill] (1.2,-2.8) circle [radius=0.1];	
	\draw[fill] (3.2,-2.8) circle [radius=0.1];	
	\draw[fill] (5.2,-2.8) circle [radius=0.1];	
	\draw[fill] (7.2,-2.8) circle [radius=0.1];	

	\draw[fill] (-2.2,-4.2) circle [radius=0.1];	
	\draw[fill] (-0.2,-4.2) circle [radius=0.1];	
	\draw[fill] (1.8,-4.2) circle [radius=0.1];	
	\draw[fill] (3.8,-4.2) circle [radius=0.1];	
		
	\draw[fill] (-3.6,-5.6) circle [radius=0.1];	
	\draw[fill] (-1.6,-5.6) circle [radius=0.1];	
	\draw[fill] (0.4,-5.6) circle [radius=0.1];

	\draw[fill] (-1.4,-1.4) circle [radius=0.1];
	\draw[fill] (-2.8,-2.8) circle [radius=0.1];
	\draw[fill] (-4.2,-4.2) circle [radius=0.1];
	\draw[fill] (-5.6,-5.6) circle [radius=0.1];

	\draw [blue,arr] (0,0) -- (-1.4,-1.4);
	\draw [blue,arr] (-1.4,-1.4) -- (-2.8,-2.8);
	\draw [blue,arr] (-2.8,-2.8) -- (-4.2,-4.2);
	\draw [blue,arr] (-4.2,-4.2) -- (-5.6,-5.6);

	\draw [blue,arr] (2,0) -- (0.6,-1.4);
	\draw [blue,arr] (0.6,-1.4) -- (-0.8,-2.8);
	\draw [blue,arr] (-0.8,-2.8) -- (-2.2,-4.2);
	\draw [blue,arr] (-2.2,-4.2) -- (-3.6,-5.6);

	\draw [blue,arr] (4,0) -- (2.6,-1.4);
	\draw [blue,arr] (2.6,-1.4) -- (1.2,-2.8);
	\draw [blue,arr] (1.2,-2.8) -- (-0.2,-4.2);
	\draw [blue,arr] (-0.2,-4.2) -- (-1.6,-5.6);

	\draw [blue,arr] (6,0) -- (4.6,-1.4);
	\draw [blue,arr] (4.6,-1.4) -- (3.2,-2.8);
	\draw [blue,arr] (3.2,-2.8) -- (1.8,-4.2);
	\draw [blue,arr] (1.8,-4.2) -- (0.4,-5.6);

	\draw [blue,arr] (8,0) -- (6.6,-1.4);
	\draw [blue,arr] (6.6,-1.4) -- (5.2,-2.8);
	\draw [blue,arr] (5.2,-2.8) -- (3.8,-4.2);

	\draw [blue,arr] (10,0) -- (8.6,-1.4);
	\draw [blue,arr] (8.6,-1.4) -- (7.2,-2.8);

	\draw [blue,arr] (12,0) -- (10.6,-1.4);

	\draw [blue,arr] (0,2) -- (2,0);
	
	\draw [blue,arr] (0,4) -- (2,2);
	\draw [blue,arr] (2,2) -- (4,0);

	\draw [blue,arr] (0,6) -- (2,4);
	\draw [blue,arr] (2,4) -- (4,2);
	\draw [blue,arr] (4,2) -- (6,0);

	\draw [blue,arr] (2,6) -- (4,4);
	\draw [blue,arr] (4,4) -- (6,2);
	\draw [blue,arr] (6,2) -- (8,0);

	\draw [blue,arr] (4,6) -- (6,4);
	\draw [blue,arr] (6,4) -- (8,2);
	\draw [blue,arr] (8,2) -- (10,0);

	\draw [blue,arr] (6,6) -- (8,4);
	\draw [blue,arr] (8,4) -- (10,2);
	\draw [blue,arr] (10,2) -- (12,0);
	
	\draw [blue,arr] (8,6) -- (10,4);
	\draw [blue,arr] (10,4) -- (12,2);
	\draw [blue,arr] (12,2) -- (14,0);
\end{tikzpicture}
  \caption{''Ground'' crystal paths of color $2$, $a=3$ and $b=4$.}
\end{figure}
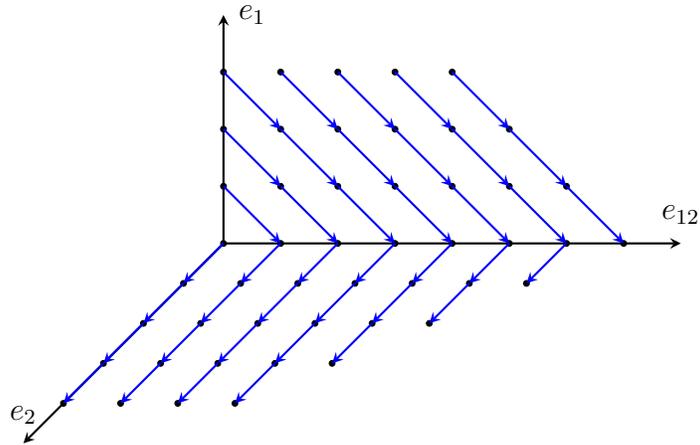

\begin{proposition}\label{crystal>}
The edge-colored graph $B^{>}(a, b)$ is an FFLV-crystal graph.
\end{proposition}
The corresponding FFLV-crystals graphs $B^>(1,1)$ and $B^>(2,2)$ are depicted below. One can verify immediately that they are indeed FFLV-crystals graphs.

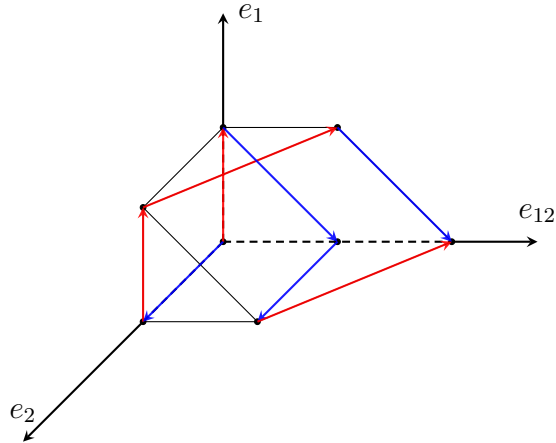
\begin{figure}
\begin{tikzpicture}[scale=.38, arr/.style={thick,->,>=stealth},dsh/.style={thick,densely dashed},]
   
    \node at (-7,-6) {{$e_2$}};    
    \node at (1,8) {{$e_1$}};
    \node at (11,1) {{$e_{12}$}};
	\draw [dsh] (0,0) -- (0,4);
	\draw [dsh] (0,0) -- (8,0);
	\draw [dsh] (0,0) -- (-2.8,-2.8);
	\draw [arr] (8,0) -- (11,0);
	\draw [arr] (0,4) -- (0,8);
	\draw [arr] (-2.8,-2.8) -- (-7,-7);
	\draw (0,4) -- (4,4) -- (-2.8,1.2) -- (0,4);
	\draw (-2.8,1.2) -- (-2.8,-2.8) -- (1.2,-2.8) -- (-2.8,1.2);
	\draw (4,4) -- (8,0) -- (1.2,-2.8);
	
	\draw[fill] (0,0) circle [radius=0.1];	
	\draw[fill] (0,4) circle [radius=0.1];	
	\draw[fill] (4,0) circle [radius=0.1];	
	\draw[fill] (4,4) circle [radius=0.1];	
	\draw[fill] (-2.8,1.2) circle [radius=0.1];	
	\draw[fill] (-2.8,-2.8) circle [radius=0.1];	
	\draw[fill] (8,0) circle [radius=0.1];	
	\draw[fill] (1.2,-2.8) circle [radius=0.1];	
	\draw [red,arr,opacity=0.9] (-2.8,-2.8) -- (-2.8,1.2);
	\draw [red,arr,opacity=0.9] (0,0) -- (0,4);
	\draw [blue,arr,opacity=0.9] (4,0) -- (1.2,-2.8);
	\draw [red,arr,opacity=0.9] (1.2,-2.8) -- (8,0);

	\draw [blue,arr,opacity=0.9] (0,0) -- (-2.8,-2.8);
	\draw [red,arr,opacity=0.9] (-2.8,1.2) -- (4,4);
	\draw [blue,arr,opacity=0.9] (0,4) -- (4,0);
	\draw [blue,arr,opacity=0.9] (4,4) -- (8,0);
\end{tikzpicture}
  \caption{The crystal $B^>(1,1)$ with respect to $\varpi_1>\varpi_2$.}
\end{figure}

\begin{figure}[H]
\begin{tikzpicture}[scale=.38, arr/.style={thick,->,>=stealth},dsh/.style={thick,densely dashed},]
   
    \node at (-8,-7) {{$e_2$}};    
    \node at (1,11) {{$e_1$}};
    \node at (18,1) {{$e_{12}$}};
	\draw [dsh] (0,0) -- (0,8);
	\draw [dsh] (0,0) -- (16,0);
	\draw [dsh] (0,0) -- (-5.6,-5.6);
	\draw [arr] (16,0) -- (18,0);
	\draw [arr] (0,8) -- (0,11);
	\draw [arr] (-5.6,-5.6) -- (-8,-8);
	\draw (0,8) -- (8,8) -- (-5.6,2.4) -- (0,8);
	\draw (-5.6,2.4) -- (-5.6,-5.6) -- (2.4,-5.6) -- (-5.6,2.4);
	\draw (8,8) -- (16,0) -- (2.4,-5.6);

\draw[fill] (0,0) circle [radius=0.1];
\draw[fill] (4,0) circle [radius=0.1];
\draw[fill] (8,0) circle [radius=0.1];
\draw[fill] (12,0) circle [radius=0.1];
\draw[fill] (16,0) circle [radius=0.1];
\draw[fill] (-2.8,-2.8) circle [radius=0.1];
\draw[fill] (1.2,-2.8) circle [radius=0.1];
\draw[fill] (5.2,-2.8) circle [radius=0.1];
\draw[fill] (9.2,-2.8) circle [radius=0.1];
\draw[fill] (-5.6,-5.6) circle [radius=0.1];
\draw[fill] (-1.6,-5.6) circle [radius=0.1];
\draw[fill] (2.4,-5.6) circle [radius=0.1];
\draw[fill] (0,4) circle [radius=0.1];
\draw[fill] (4,4) circle [radius=0.1];
\draw[fill] (8,4) circle [radius=0.1];
\draw[fill] (12,4) circle [radius=0.1];
\draw[fill] (-2.8,1.2) circle [radius=0.1];
\draw[fill] (1.2,1.2) circle [radius=0.1];
\draw[fill] (5.2,1.2) circle [radius=0.1];
\draw[fill] (-5.6,-1.6) circle [radius=0.1];
\draw[fill] (-1.6,-1.6) circle [radius=0.1];
\draw[fill] (0,8) circle [radius=0.1];
\draw[fill] (4,8) circle [radius=0.1];
\draw[fill] (8,8) circle [radius=0.1];
\draw[fill] (-2.8,5.2) circle [radius=0.1];
\draw[fill] (1.2,5.2) circle [radius=0.1];
\draw[fill] (-5.6,2.4) circle [radius=0.1];

	\draw [red,arr,opacity=0.9] (0,0) -- (0,4);
	\draw [red,arr,opacity=0.9] (0,4) -- (0,8);
	
	\draw [red,arr,opacity=0.9] (-2.8,1.2) -- (-2.8,5.2);
	\draw [red,arr,opacity=0.9] (-2.8,-2.8) -- (-2.8,1.2);
	
	\draw [red,arr,opacity=0.9] (-5.6,-5.6) -- (-5.6,-1.6);
	\draw [red,arr,opacity=0.9] (-5.6,-1.6) -- (-5.6,2.4);
	
	\draw [red,arr,opacity=0.9] (4,0) -- (4,4);
	\draw [red,arr,opacity=0.9] (1.2,-2.8) -- (1.2,1.2);
	\draw [red,arr,opacity=0.9] (-1.6,-5.6) -- (-1.6,-1.6);
	
	\draw [red,arr,opacity=0.9]  (-5.6,2.4) -- (1.2,5.2);
	\draw [red,arr,opacity=0.9]  (1.2,5.2) -- (8,8);
	
	\draw [red,arr,opacity=0.9]  (-2.8,5.2) -- (4,8);
	
	\draw [red,arr,opacity=0.9]  (2.4,-5.6) -- (9.2,-2.8);
	\draw [red,arr,opacity=0.9]  (9.2,-2.8) -- (16,0);
	
	\draw [red,arr,opacity=0.9]  (5.2,-2.8) -- (12,0);
	
	\draw [red,arr,opacity=0.9] (1.2,1.2) -- (8,4);	

	\draw [red,arr,opacity=0.9] (-1.6,-1.6) -- (5.2,1.2);
	\draw [red,arr,opacity=0.9] (5.2,1.2) -- (12,4);

	\draw [blue,arr,opacity=0.9] (8,8) -- (12,4);
	\draw [blue,arr,opacity=0.9] (12,4) -- (16,0);
	\draw [blue,arr,opacity=0.9] (4,8) -- (8,4);
	\draw [blue,arr,opacity=0.9] (8,4) -- (12,0);
	\draw [blue,arr,opacity=0.9] (0,8) -- (4,4);
	\draw [blue,arr,opacity=0.9] (4,4) -- (8,0);
	\draw [blue,arr,opacity=0.9] (0,4) -- (4,0);
	\draw [blue,arr,opacity=0.9] (0,0) -- (-2.8,-2.8);
	\draw [blue,arr,opacity=0.9] (-2.8,-2.8) -- (-5.6,-5.6);
	\draw [blue,arr,opacity=0.9] (4,0) -- (1.2,-2.8);
	\draw [blue,arr,opacity=0.9] (1.2,-2.8) -- (-1.6,-5.6);
	\draw [blue,arr,opacity=0.9] (8,0) -- (5.2,-2.8);
	\draw [blue,arr,opacity=0.9] (5.2,-2.8) -- (2.4,-5.6);
	\draw [blue,arr,opacity=0.9] (12,0) -- (9.2,-2.8);

	\draw [blue,arr,opacity=0.9] (1.2,5.2) -- (5.2,1.2);
	\draw [blue,arr,opacity=0.9] (-2.8,1.2) -- (-5.6,-1.6);
	\draw [blue,arr,opacity=0.9] (-2.8,5.2) -- (1.2,1.2);
	\draw [blue,arr,opacity=0.9] (1.1,1.3) -- (-1.75,-1.45);

\end{tikzpicture}
  \caption{The crystal $B^>(2,2)$ with respect to $\varpi_1>\varpi_2$.}
\end{figure}
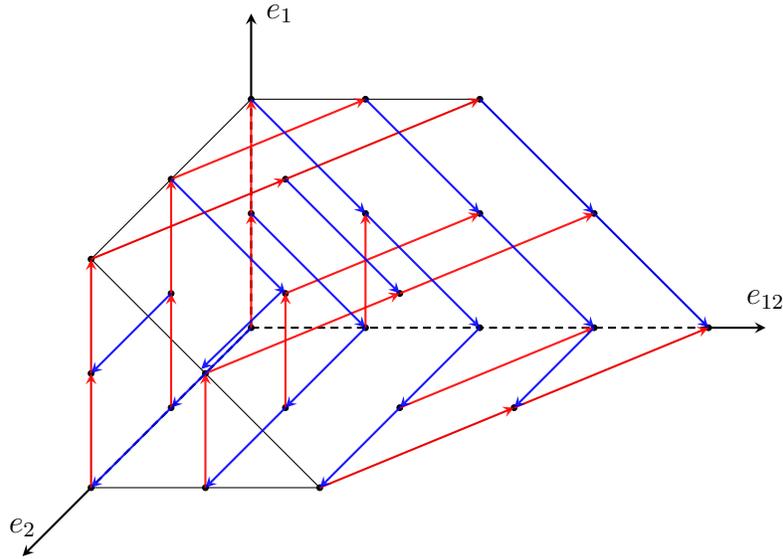

\begin{proof}[Proof of Proposition \ref{crystal>}]
For the proof we take the crystal graph $K(a, b)$ constructed in \cite[Theorem 3.1]{a2}  which satisfies (A1)-(A4) and establish a crystal bijection $\kappa: K(a, b)\to B^{>}(a, b)$.

Let $H$ be the linear hyperplane in $\mathbb{R}^{\Delta_+}$ having normal vector $(1,0,-1)$ and denote by $C(a, b)$ the set of lattice points in the intersection
\[
\FFLV_2(a\varpi_1+b\varpi_2)\cap H.
\]

Without loss of generality we assume that $b\geq a$ (the other case can be treated similarly). 

The set $C(a, b)$ is the set of critical points (for the definition see \cite{a2}) in $B^{>}(a, b)$, which has cardinality $(a+1)(b+1)$: it is constituted from the lattice points of union the rectangular of size $a\times (b-a)$ and the half of the rectangular of size $2a\times a$, along the common edge of length $a$.

The crystal $K(a,b)$ has the same amount of critical points (\cite[Corollary 3.2]{a2}). 

By the construction of $K(a,b)$ in the beginning of \cite[Section 3]{a2}, the critical points in $K(a,b)$ can be identified with the integral points in the rectangular $[0,a]\times [0,b]$. The map $\kappa$ sends the corners of the rectangle to the following corners in $C(a, b)$: $(0,0,0)$, $(0, a, 0)$, $(0, a+b, 0)$, and $(a, b-a, a)$ (precisely, $\kappa^{-1}(0, a+b, 0)$ is opposite to $\kappa^{-1}(0,0,0)$ and   $\kappa^{-1}(a, b-a, a)$ is opposite to $\kappa^{-1}(0,a,0)$). 

We first define the image of the $(b+1)$-copies (labeled by $0,1,\cdots,b$) of $K(a,0)$ under the map $\kappa$. The image of the $0$-th copy of $K(a,0)$ is the subgraph of $B^{>}(a, b)$ bounded by the path from $(0,0,0)$ to $(a,0,0)$ of color $1$, the path from $(a, 0, 0)$ to $(0, a, 0)$ of color $2$, and the critical points on the segment from $(0,0,0)$ to $(0,a,0)$. The image of the $m$-th copy, where $1\leq m\le b$, is obtained as follows. We denote by $\pi_m$ the path in the set of "sky" crystal paths emanating from $(0,0,m)$. Consider the part of $\pi_m$ between its critical point and its endpoint: such a 1-color path, denoted by $\pi_m^+$, has length $a$. To each vertex in $\pi_m^+$, we engraft the part of the path of color $2$ emanating from this vertex until its critical point. The subgraph of $B^>(a, b)$ with this set of vertices at these parts of paths is the image of the $m$-th copy of $K(a,0)$ under $\kappa$. 
 
 It is easy to see from this construction that such defined subgraphs are isomorphic to $K(a,0)$ and they cover the set of critical points properly without overlapping.

We define the images of the $(a+1)$-copies (labeled by $0,1,\cdots,a$) of $K(0,b)$ under the map $\kappa$. For the image of the $0$-th copy of $K(0,b)$, we take the part of the path $\pi_b$ from its beginning to the critical point $(a,b-a,a)$; to each vertex in this path, we consider the path of color $2$ terminating at this vertex and take the part of such a path between its critical point and the endpoint. We finally take the subgraph of $B^>(a, b)$ with vertices on all such defined paths. This is the image of the $0$-th copy of $K(0,b)$. 

For the image of the $k$-th copy, where $1\leq k\leq a$, we consider the translation of the $1$-colored path $\pi_b$ by the vector $(-k,k,0)$ and take from it the part starting from its beginning to its critical point: such a path has length equal to $b$. To each vertex on such a path, we engraft the part of the 2-colored path terminating at this vertex from its critical point. The sub-crystal with such defined set of vertices is the image of the $k$-th copy $K(0,b)$. Since the 2-colored paths appearing in the construction are translations of those in the "ground" crystal paths, we get that such a sub-crystal is isomorphic to $K(0,b)$.  Note that these images of copies $K(0, b)$ cover the set of critical points properly without overlapping.

According to \cite[Theorem 3.1]{a2}, such defined map $\kappa$ is a crystal bijection, the proof terminates.
\end{proof}

\begin{remark} 
We can consider a kind of $B(\infty)$ crystal of the above form by sending $a$ and $b$ to $+\infty$. Then we will get the crystal graph on lattice points of the positive orthant, which has the monochromatic paths of color $1$ being vertical rays $\bz+\mathbb{R} e_1$, and the monochromatic path of color 2 being translations of the ''ground'' paths with $a=b=+\infty$.
Unfortunately, we can not embed any of FFLV-crystals to such kind of $B(\infty)$.
\end{remark}

The case $\sigma=(1,2)\in\mathfrak{S}_2$ corresponding to $\varpi_2>\varpi_1$. It can be treated similarly. Let $B^{<}(a, b)$ be an edge-colored graph on lattice points in $\FFLV_2(a\varpi_1+b\varpi_2)_\mathbb{Z}$ such that its monochromatic paths are defined in the following way:

\begin{enumerate}
\item[(1)]  For the color 1, we take the ''ground'' paths depicted below and their translations by the vector $(k,0,k)$, $k=1,\cdots$ (precisely we takes the parts of translated paths which belong to $\FFLV_2(a,b)$). 
\item[(2)] For the color 2, we take the ''wall'' paths depicted below and their translations by the vector $(0,k,-k)$, $k=1,\cdots$ (precisely we takes the parts of translated paths which belong to $\FFLV_2(a,b)$).
\end{enumerate}

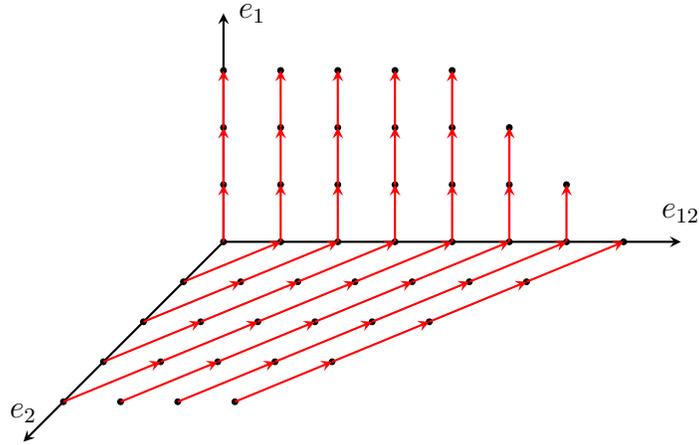
\begin{figure}[H]
\begin{tikzpicture}[scale=.38, arr/.style={thick,->,>=stealth},dsh/.style={thick,densely dashed},]   
    \node at (-7,-6) {{$e_2$}};    
    \node at (1,8) {{$e_1$}};
    \node at (16,1) {{$e_{12}$}};
	\draw [arr] (0,0) -- (16,0);
	\draw [arr] (0,0) -- (0,8);
	\draw [arr] (0,0) -- (-7,-7);

	\draw[fill] (0,0) circle [radius=0.1];	
	\draw[fill] (2,0) circle [radius=0.1];	
	\draw[fill] (4,0) circle [radius=0.1];	
	\draw[fill] (6,0) circle [radius=0.1];	
	\draw[fill] (8,0) circle [radius=0.1];	
	\draw[fill] (10,0) circle [radius=0.1];	
	\draw[fill] (12,0) circle [radius=0.1];	
	\draw[fill] (14,0) circle [radius=0.1];	

	\draw[fill] (0,2) circle [radius=0.1];	
	\draw[fill] (2,2) circle [radius=0.1];	
	\draw[fill] (4,2) circle [radius=0.1];	
	\draw[fill] (6,2) circle [radius=0.1];	
	\draw[fill] (8,2) circle [radius=0.1];	
	\draw[fill] (10,2) circle [radius=0.1];	
	\draw[fill] (12,2) circle [radius=0.1];	

	\draw[fill] (0,4) circle [radius=0.1];	
	\draw[fill] (2,4) circle [radius=0.1];	
	\draw[fill] (4,4) circle [radius=0.1];	
	\draw[fill] (6,4) circle [radius=0.1];	
	\draw[fill] (8,4) circle [radius=0.1];	
	\draw[fill] (10,4) circle [radius=0.1];	
	
	\draw[fill] (0,6) circle [radius=0.1];	
	\draw[fill] (2,6) circle [radius=0.1];	
	\draw[fill] (4,6) circle [radius=0.1];	
	\draw[fill] (6,6) circle [radius=0.1];	
	\draw[fill] (8,6) circle [radius=0.1];	
	
	\draw[fill] (-1.4,-1.4) circle [radius=0.1];	
	\draw[fill] (0.6,-1.4) circle [radius=0.1];	
	\draw[fill] (2.6,-1.4) circle [radius=0.1];	
	\draw[fill] (4.6,-1.4) circle [radius=0.1];	
	\draw[fill] (6.6,-1.4) circle [radius=0.1];	
	\draw[fill] (8.6,-1.4) circle [radius=0.1];	
	\draw[fill] (10.6,-1.4) circle [radius=0.1];	
	
	\draw[fill] (-2.8,-2.8) circle [radius=0.1];	
	\draw[fill] (-0.8,-2.8) circle [radius=0.1];	
	\draw[fill] (1.2,-2.8) circle [radius=0.1];	
	\draw[fill] (3.2,-2.8) circle [radius=0.1];	
	\draw[fill] (5.2,-2.8) circle [radius=0.1];	
	\draw[fill] (7.2,-2.8) circle [radius=0.1];	
	
	\draw[fill] (-4.2,-4.2) circle [radius=0.1];	
	\draw[fill] (-2.2,-4.2) circle [radius=0.1];	
	\draw[fill] (-0.2,-4.2) circle [radius=0.1];	
	\draw[fill] (1.8,-4.2) circle [radius=0.1];	
	\draw[fill] (3.8,-4.2) circle [radius=0.1];	
	
	\draw[fill] (-5.6,-5.6) circle [radius=0.1];	
	\draw[fill] (-3.6,-5.6) circle [radius=0.1];	
	\draw[fill] (-1.6,-5.6) circle [radius=0.1];	
	\draw[fill] (0.4,-5.6) circle [radius=0.1];	

	\draw [red,arr] (0.4,-5.6) -- (3.8,-4.2);
	\draw [red,arr] (3.8,-4.2) -- (7.2,-2.8);
	\draw [red,arr] (7.2,-2.8) -- (10.6,-1.4);
	\draw [red,arr] (10.6,-1.4) -- (14,0);

	\draw [red,arr] (-1.6,-5.6) -- (1.8,-4.2);
	\draw [red,arr] (1.8,-4.2) -- (5.2,-2.8);
	\draw [red,arr] (5.2,-2.8) -- (8.6,-1.4);
	\draw [red,arr] (8.6,-1.4) -- (12,0);

	\draw [red,arr] (-3.6,-5.6) -- (-0.2,-4.2);
	\draw [red,arr] (-0.2,-4.2) -- (3.2,-2.8);
	\draw [red,arr] (3.2,-2.8) -- (6.6,-1.4);
	\draw [red,arr] (6.6,-1.4) -- (10,0);	

	\draw [red,arr] (-5.6,-5.6) -- (-2.2,-4.2);
	\draw [red,arr] (-2.2,-4.2) -- (1.2,-2.8);
	\draw [red,arr] (1.2,-2.8) -- (4.6,-1.4);
	\draw [red,arr] (4.6,-1.4) -- (8,0);	

	\draw [red,arr] (-4.2,-4.2) -- (-0.8,-2.8);
	\draw [red,arr] (-0.8,-2.8) -- (2.6,-1.4);
	\draw [red,arr] (2.6,-1.4) -- (6,0);	

	\draw [red,arr] (-2.8,-2.8) -- (0.6,-1.4);
	\draw [red,arr] (0.6,-1.4) -- (4,0);	
	
	\draw [red,arr] (-1.4,-1.4) -- (2,0);	
	
	\draw [red,arr] (0,0) -- (0,2);	
	\draw [red,arr] (0,2) -- (0,4);	
	\draw [red,arr] (0,4) -- (0,6);	

	\draw [red,arr] (2,0) -- (2,2);	
	\draw [red,arr] (2,2) -- (2,4);	
	\draw [red,arr] (2,4) -- (2,6);	

	\draw [red,arr] (4,0) -- (4,2);	
	\draw [red,arr] (4,2) -- (4,4);	
	\draw [red,arr] (4,4) -- (4,6);	

	\draw [red,arr] (6,0) -- (6,2);	
	\draw [red,arr] (6,2) -- (6,4);	
	\draw [red,arr] (6,4) -- (6,6);	

	\draw [red,arr] (8,0) -- (8,2);	
	\draw [red,arr] (8,2) -- (8,4);	
	\draw [red,arr] (8,4) -- (8,6);	

	\draw [red,arr] (10,0) -- (10,2);	
	\draw [red,arr] (10,2) -- (10,4);	
	
	\draw [red,arr] (12,0) -- (12,2);		
\end{tikzpicture}
  \caption{''Ground'' crystal paths of color $1$, $a=3$ and $b=4$.}
\end{figure}

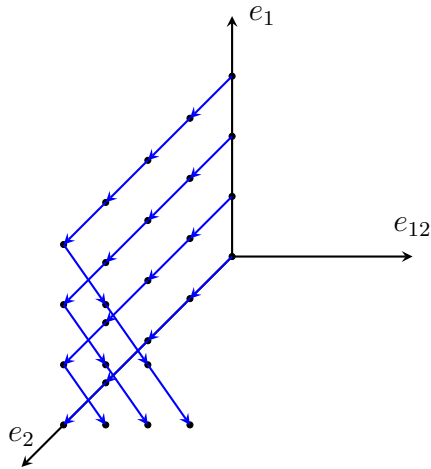
\begin{figure}[H]
\begin{tikzpicture}[scale=.4, arr/.style={thick,->,>=stealth},dsh/.style={thick,densely dashed},]
   
    \node at (-7,-6) {{$e_2$}};    
    \node at (1,8) {{$e_1$}};
    \node at (6,1) {{$e_{12}$}};
	\draw [arr] (0,0) -- (6,0);
	\draw [arr] (0,0) -- (0,8);
	\draw [arr] (0,0) -- (-7,-7);

	\draw[fill] (0,0) circle [radius=0.1];	
	\draw[fill] (0,2) circle [radius=0.1];	
	\draw[fill] (0,4) circle [radius=0.1];	
	\draw[fill] (0,6) circle [radius=0.1];	

	\draw[fill] (-1.4,-1.4) circle [radius=0.1];	
	\draw[fill] (-1.4,0.6) circle [radius=0.1];	
	\draw[fill] (-1.4,2.6) circle [radius=0.1];	
	\draw[fill] (-1.4,4.6) circle [radius=0.1];	

	\draw[fill] (-2.8,-2.8) circle [radius=0.1];	
	\draw[fill] (-2.8,-0.8) circle [radius=0.1];	
	\draw[fill] (-2.8,1.2) circle [radius=0.1];	
	\draw[fill] (-2.8,3.2) circle [radius=0.1];	

	\draw[fill] (-4.2,-4.2) circle [radius=0.1];	
	\draw[fill] (-4.2,-2.2) circle [radius=0.1];	
	\draw[fill] (-4.2,-0.2) circle [radius=0.1];	
	\draw[fill] (-4.2,1.8) circle [radius=0.1];	

	\draw[fill] (-5.6,-5.6) circle [radius=0.1];	
	\draw[fill] (-5.6,-3.6) circle [radius=0.1];	
	\draw[fill] (-5.6,-1.6) circle [radius=0.1];	
	\draw[fill] (-5.6,0.4) circle [radius=0.1];

	\draw [blue,arr] (0,0) -- (-1.4,-1.4);
	\draw [blue,arr] (-1.4,-1.4) -- (-2.8,-2.8);
	\draw [blue,arr] (-2.8,-2.8) -- (-4.2,-4.2);
	\draw [blue,arr] (-4.2,-4.2) -- (-5.6,-5.6);

	\draw [blue,arr] (0,2) -- (-1.4,0.6);
	\draw [blue,arr] (-1.4,0.6) -- (-2.8,-0.8);
	\draw [blue,arr] (-2.8,-0.8) -- (-4.2,-2.2);
	\draw [blue,arr] (-4.2,-2.2) -- (-5.6,-3.6);

	\draw [blue,arr] (0,4) -- (-1.4,2.6);
	\draw [blue,arr] (-1.4,2.6) -- (-2.8,1.2);
	\draw [blue,arr] (-2.8,1.2) -- (-4.2,-0.2);
	\draw [blue,arr] (-4.2,-0.2) -- (-5.6,-1.6);
	
	\draw [blue,arr] (0,6) -- (-1.4,4.6);
	\draw [blue,arr] (-1.4,4.6) -- (-2.8,3.2);
	\draw [blue,arr] (-2.8,3.2) -- (-4.2,1.8);
	\draw [blue,arr] (-4.2,1.8) -- (-5.6,0.4);

	\draw[fill] (-4.2,-1.6) circle [radius=0.1];	
	\draw[fill] (-2.8,-3.6) circle [radius=0.1];	
	\draw[fill] (-1.4,-5.6) circle [radius=0.1];

	\draw[fill] (-4.2,-3.6) circle [radius=0.1];	
	\draw[fill] (-2.8,-5.6) circle [radius=0.1];	

	\draw[fill] (-4.2,-5.6) circle [radius=0.1];	
	
	\draw [blue,arr] (-5.6,0.4) -- (-4.2,-1.6);
	\draw [blue,arr] (-4.2,-1.6) -- (-2.8,-3.6);
	\draw [blue,arr] (-2.8,-3.6) -- (-1.4,-5.6);

	\draw [blue,arr] (-5.6,-1.6) -- (-4.2,-3.6);
	\draw [blue,arr] (-4.2,-3.6) -- (-2.8,-5.6);

	\draw [blue,arr] (-5.6,-3.6) -- (-4.2,-5.6);
		
\end{tikzpicture}
  \caption{''Wall'' crystal paths of color $2$, $a=3$ and $b=4$.}
\end{figure}

\begin{figure}[H]
\begin{tikzpicture}[scale=.38, arr/.style={thick,->,>=stealth},dsh/.style={thick,densely dashed},]
   
    \node at (-7,-6) {{$e_2$}};    
    \node at (1,8) {{$e_1$}};
    \node at (11,1) {{$e_{12}$}};
	\draw [dsh] (0,0) -- (0,4);
	\draw [dsh] (0,0) -- (8,0);
	\draw [dsh] (0,0) -- (-2.8,-2.8);
	\draw [arr] (8,0) -- (11,0);
	\draw [arr] (0,4) -- (0,8);
	\draw [arr] (-2.8,-2.8) -- (-7,-7);
	\draw (0,4) -- (4,4) -- (-2.8,1.2) -- (0,4);
	\draw (-2.8,1.2) -- (-2.8,-2.8) -- (1.2,-2.8) -- (-2.8,1.2);
	\draw (4,4) -- (8,0) -- (1.2,-2.8);
	
	\draw[fill] (0,0) circle [radius=0.1];	
	\draw[fill] (0,4) circle [radius=0.1];	
	\draw[fill] (4,0) circle [radius=0.1];	
	\draw[fill] (4,4) circle [radius=0.1];	
	\draw[fill] (-2.8,1.2) circle [radius=0.1];	
	\draw[fill] (-2.8,-2.8) circle [radius=0.1];	
	\draw[fill] (8,0) circle [radius=0.1];	
	\draw[fill] (1.2,-2.8) circle [radius=0.1];	

	\draw [red,arr,opacity=0.9] (0,0) -- (0,4);
	\draw [blue,arr,opacity=0.9] (-2.8,1.2) -- (1.2,-2.8);
	\draw [red,arr,opacity=0.9] (-2.8,-2.8) -- (4,0);
	\draw [red,arr,opacity=0.9] (1.2,-2.8) -- (8,0);

	\draw [blue,arr,opacity=0.9] (0,0) -- (-2.8,-2.8);
	\draw [blue,arr,opacity=0.9] (0,4) -- (-2.8,1.2);
	\draw [red,arr,opacity=0.9] (4,0) -- (4,4);
	\draw [blue,arr,opacity=0.9] (4,4) -- (8,0);

\end{tikzpicture}
  \caption{The crystal $B^<(1,1)$ with respect to $\varpi_2>\varpi_1$.}
\end{figure}
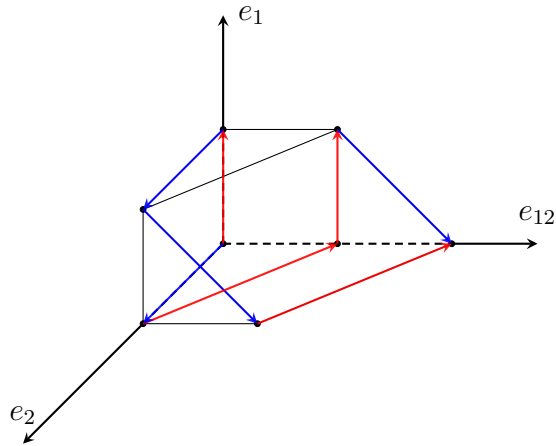

\begin{proposition}\label{crystal>}
The edge-colored graph $B^{<}(a, b)$ is an FFLV-crystal graph.
\end{proposition}

Similarly to the proof of Proposition \ref{crystal>}, one can establish a crystal bijection $\kappa':K(a, b)\to B^{<}(a, b)$.
We leave details to the reader.


\begin{thebibliography}{99}

\bibitem{ABS} F.~Ardila, T.~Bliem, and D.~Salazar,
 \emph{Gelfand-Tsetlin polytopes and Feigin-Fourier-Littelmann-Vinberg polytopes as marked poset polytopes.} J. Combin. Theory Ser. A 118 (2011), no. 8, 2454--2462.
 
\bibitem{BZ01}
{ A. Berenstein, A. Zelevinsky,}
{\em Tensor product multiplicities, canonical bases and totally positive varieties,} 
Invent. Math. 143 (2001), 77--128.

\bibitem{BF}
{ L. Bossinger, G. Fourier,}
\emph{String cone and Superpotential combinatorics for flag and Schubert varieties in type $\mathsf{A}$},
J. Combin. Theory Ser. A., Volume 167, October 2019, pp 213--256.
 
\bibitem{Cal1} 
{ P. Caldero}, 
{\em A multiplicative property of quantum flag minors},
{Representation Theory}, {\bf 7}, 164-176 (2003). 


\bibitem{a2} { V. I. Danilov,  A. V. Karzanov,  G. A. Koshevoy},  \emph{Combinatorics of regular $A_2$-crystals}, Journal of Algebra 310 (2007) 218--234

\bibitem{E}
{ S. Elnitsky}, 
\emph{Rhombic tilings of polygons and classes of reduced words in Coxeter groups}, J. Combin. Theory Ser. A, 77(2):193--221, 1997.

\bibitem{FF}
X. Fang, G. Fourier,
\emph{Symmetries on plabic graphs and associated polytopes},  Comptes Rendus Math\'ematique, Volume 356, Issue 6, Pages 581-585, June 2018.

\bibitem{FaFoL}
X. Fang, G. Fourier and P. Littelmann,
\emph{Essential bases and toric degenerations arising from birational sequences,} 
Adv. Math, Volume 312, 25 May 2017, Pages 107-149.

\bibitem{FaFoR}
X. Fang, G. Fourier and M. Reineke,
\emph{PBW-Type filtration on quantum groups of type $\mathsf{A}_n$},
Journal of Algebra 449 (2016) 321--345.
 
\bibitem{Fei} 
E. Feigin, 
\emph{$\mathbb{G}^M_a$ degeneration of flag varieties}. Selecta Math. (N.S.), 18 (2012), no. 3, 513--537.

\bibitem{FFL1} 
E. Feigin, G. Fourier, P. Littelmann, 
\emph{PBW filtration and bases for irreducible modules in type ${\mathsf{A}}_n$}, Transform. Groups 16 (2011), no. 1, 71--89.

\bibitem{FFL2} E. Feigin, G. Fourier, P. Littelmann, 
\emph{PBW filtration and bases for symplectic Lie algebras}. Int. Math. Res. Not. IMRN 2011, no. 24, 5760--5784.

\bibitem{Fou}
G.~Fourier, \emph{Marked poset polytopes: Minkowski sums, indecomposables, and unimodular equivalence},
Journal of Pure and Applied Algebra, 220, Issue 2, 2016,  606--620.

\bibitem{GKS16}
V. Genz, G. Koshevoy and B. Schumann, 
\emph{Combinatorics of canonical bases revisited: Type $\mathsf{A}$},
preprint, arXiv:1611.03465.


\bibitem{GKS19}
V. Genz, G. Koshevoy and B. Schumann, 
\emph{On the interplay of the parametrizations of canonical bases by Lusztig and string data},
preprint, arXiv:1901.03500.

\bibitem{Jantzen}
J. C.~Jantzen,
\emph{Lectures on quantum groups.}
Graduate Studies in Mathematics, 6. American Mathematical Society, Providence, RI, 1996.

\bibitem{Kas91}
M. Kashiwara,
\emph{On crystal bases of the q-analogue of universal enveloping algebras,}
Duke Math. J. 63, (1991), 465--516.

\bibitem{KK} 
K.~Kaveh, A. G.~Khovanskii, 
\emph{Newton-Okounkov bodies, semigroups of integral points, graded algebras and intersection theory}, Ann. of Math. (2) 176 (2012), no. 2, 925--978.


\bibitem{Ko}
G. Koshevoy,
\emph{Cluster decorated geometric crystals, generalized geometric RSK-correspondences, and Donaldson-Thomas transformations}. 2017 MATRIX annals, 363--387.

\bibitem{Kus}
D. Kus,
\emph{Realization of affine type A Kirillov–Reshetikhin crystals via polytopes,}
Journal of Combinatorial Theory, Series A 120 (2013) 2093--2117.


\bibitem{LaMu}
R. Lazarsfeld and M. Musta\c{t}$\breve{\rm a}$,
\emph{Convex bodies associated to linear series},
Ann. Sci. \'Ec. Norm. Sup\'er. (4) 42 (2009), no. 5, 783--835. 


\bibitem{Lus90}
{ G. Lusztig,}
{\em Canonical bases arising from quantized enveloping algebras},
J. Amer. Math. Soc. 3 (1990), no. 2, 447--498.

\bibitem{Lus92}
{ G. Lusztig,}
{\em Introduction to quantized enveloping algebras,} 
New developments in Lie theory and their applications, ed. J.Tirao, Progr. in Math., vol. 105, Birkh\"auser Boston, 1992, pp. 49--65.

\bibitem{lusztigbook}
{ G. Lusztig,}
{\em Introduction to quantum groups},
Progress in Mathematics 110, Birkh\"auser, Boston, 1993.

\bibitem{Ma}
I. Makhlin,
\emph{Gelfand-Tsetlin degenerations of representations and flag varieties,}
preprint, arXiv:1809.02258.

\bibitem{MY}
A. Molev, O. Yakimova,
\emph{Monomial bases and branching rules,}
preprint, arXiv:1812.03698.

\bibitem{St}
J. Stembridge,
\emph{A local characterization of simply-laced crystals,}
Transactions of AMS, Volume 355, Number 12, Pages 4807--4823.

\bibitem{Sta}
R.~P. Stanley,
\emph{Two poset polytopes},
{Discrete Comput. Geom.}, 1(1):9--23, 1986.

\end{thebibliography}
\end{document}